%% file: rpcd.tex
\documentclass[onefignum,onetabnum]{siamart190516}

\input{ex_shared}

\begin{document}

\maketitle

\begin{abstract}
  This paper presents a memory efficient, first-order method for low multi-linear rank approximation of high-order, high-dimensional tensors. 
  In our method, we exploit the second-order information of the cost function and the constraints to suggest a new Riemannian metric on the Grassmann manifold.  
  We use a Riemmanian coordinate descent method for solving the problem, and also provide a global convergence analysis matching that of the coordinate descent method in the Euclidean setting. We also show that each step of our method with the unit step-size is actually a step of the orthogonal iteration algorithm. 
  Experimental results show the computational advantage of our method for high-dimensional tensors.
\end{abstract}



\section{Introduction}
Higher-order tensors are ubiquitous in factor analysis problems across multiple disciplines, including psychometrics, econometrics, biomedical signal processing, data mining, and social network analysis. 
In particular, tensor decomposition techniques with low-rank approximations offer several benefits, such as reducing the number of dimensions, removing noise, and uncovering latent variables. 
The utility of tensor decomposition has been demonstrated in a range of applications, from identifying underlying factors in psychometric studies to identifying sources of signals in biomedical research. 
To gain a comprehensive understanding of tensors and their decomposition methods, Kolda and Bader's review \cite{kolda2009tensor} provides a broad perspective, including mathematical foundations and algorithmic approaches, while Sidiropoulos et al.'s work \cite{sidiropoulos2017tensor} provides a more recent overview with a specific emphasis on signal and data analysis.

Tucker decomposition, first introduced by \cite{tucker1966some}, is a mathematical technique that involves factorizing a tensor $\mathcal{X} \in \mathbb{R}^{n_1 \times \cdots \times n_d}$ of multi-linear rank-$(r_1, ..., r_d)$ into a core tensor and $d$ factor matrices. 
This approach extends principal component analysis from matrices to tensors, as noted in \cite{zare2018extension}. 
So, in Tucker decomposition, the factor matrices can be viewed as the principal components for each mode of the tensor. 
The factorization of the tensor $\mathcal{X}$ can be expressed as follows:
\begin{equation*}
	\mathcal{X} = \mathcal{C}\times_1 U_1 \times_2 \cdots \times_d U_d \ ,
\end{equation*}
where $\mathcal{X}\times_i U_i$ is the $i$-mode product (see \cref{i-mode}) between $\mathcal{X}$ and $U_i$, $\mathcal{C} \in \mathbb{R}^{r_1\times\cdots\times r_d}$ and $U_i \in St(n_i,r_i)$ (see \cref{Stiefel}) denote the core tensor and each of the orthonormal factor matrices, respectively. 
Typically, $r_i \ll n_i$, which enables $\mathcal{C}$ to be considered as a compressed or dimensionally reduced version of the original tensor $\mathcal{X}$.

The storage complexity of Tucker decomposition is proportional to $O(\prod_{i=1}^d r_i + \sum_{i=1}^d n_ir_i)$, as opposed to $O(\prod_{i=1}^d n_i)$ for the original tensor $\mathcal{X}$. 
Once the factor matrices are found, the core tensor can be computed as
\begin{equation*}
	\mathcal{C} = \mathcal{X}\times_1 U_1^T \times_2 \cdots \times_d U_d^T \ .
\end{equation*}

The manifold $St(n_i,r_i)$, which consists of all $n_i \times r_i$ matrices with orthonormal columns, does not guarantee the uniqueness of the decomposition due to the symmetry inherent in the problem.
Any modified core tensor $\Bar{\mathcal{C}} = \mathcal{C}\times_1 Q_1^T \times_2 \cdots \times_d Q_d^T$ where $Q_i$s are orthogonal matrices, is another low-rank representation of the original tensor, i.e., $\mathcal{X} = \Bar{\mathcal{C}} \times_1 U_1 Q_1 \times_2 \cdots \times_d U_d Q_d$.
However, we can achieve uniqueness by constraining the matrices $U_i$ to lie on the Grassmann manifold, as we will demonstrate later in this paper. 
Therefore, we propose a Riemannian coordinate descent algorithm that operates on the product space of Grassmann manifolds to solve this problem.
In addition to the Grassmann manifold, other common constraints for the $U_i$ matrices include statistical independence, sparsity, and non-negativity
\cite{cichocki2009nonnegative,morup2008algorithms}.
These constraints provide prior information about the underlying factors and lead to more interpretable results.

In the manifold optimization literature \cite{absil2009optimization,boumal2020introduction}, it is a common practice to transform a constrained optimization problem in Euclidean space into an unconstrained problem on a manifold that represents the constraints. 
This approach offers several advantages over traditional constrained optimization methods.
One significant advantage of Riemannian optimization methods is that they ensures exact satisfaction of constraints at every iteration, whereas classical constrained optimization methods only approximately satisfy constraints. 
Additionally, manifold optimization respects the geometry of the problem, meaning that the definition of the inner product can provide more meaningful Riemannian gradient directions.

The approach of coordinate descent methods \cite{wright2015coordinate}, involves making partial updates to the decision variables. 
This method offers an advantage in terms of the ease of generating search directions and updating variables. This feature is particularly useful when working with large-scale problems. 
Additionally, coordinate descent methods tend to exhibit fast empirical convergence, particularly during the initial optimization steps. 
As such, they are a suitable option for approximations.

Gutman and Ho-Nguyen \cite{gutman2022coordinate} proposed an extension of the coordinate descent method that operates in the manifold domain. 
Instead of minimizing over coordinates, the authors performed inexact minimization over subspaces of the tangent space at every point.
They noted that the convergence rate in the case of product manifolds is comparable to that in the Euclidean setting. 
Drawing on this insight, we employed a coordinate descent approach to compute the factor matrices in Tucker decomposition.
Specifically, we solved an optimization problem for each factor matrix using a reformulated cost function subject to the Grassmann manifold constraint.

Gradient-based algorithms are widely used in solving large-scale problems, but sometimes they encounter convergence issues. 
To achieve better convergence rates, it is beneficial to find a suitable metric. 
Although the construction of a Riemannian metric typically focuses on the geometry of the constraints, considering the role of the cost function can also be helpful when possible. 
This approach was introduced in \cite{mishra2016riemannian}, where the second-order information of the Lagrangian was encoded into the metric. We adopt this method to develop a new metric that demonstrates excellent performance.
In addition, incorporating preconditioning into the coordinate descent algorithm in Euclidean space has also been explored in the literature (see \cite{tappenden2016inexact} for an example).

By combining all these factors, we introduce a new method, termed Riemannian Preconditioned Coordinate Descent (RPCD). The contributions of our paper are as follows:
\begin{itemize}
	\item RPCD is a first-order optimization-based algorithm which has advantages over SVD-based methods and second-order methods in large scale cases. It is also very efficient with respect to the memory complexity.
	\item We construct a Riemannian metric by using the second-order information of the cost function and constraint to solve the Tucker decomposition as a series of unconstrained problems on the Grassmann manifold.
	\item We provide a convergence analysis for the Riemannian coordinate descent algorithm in a relatively general setting. 
	This is done by modification of the convergence analysis in \cite{gutman2022coordinate} to the case of product manifolds when exponential map and parallel transport are replaced by retraction and vector transport. Our proposed RPCD algorithm for Tucker decomposition is a special case of Riemannian coordinate descent, and therefore the proofs hold for its global convergence. 
\end{itemize}
The results of our experiments, conducted on both synthetic and real data, demonstrate the superior performance of the proposed algorithm.
  
\subsection{Related work}
There are two algorithmic approaches to solving the Tucker decomposition problem. 
The first approach is based on Singular Value Decomposition (SVD) and aims to extend truncated SVD from matrices to tensors. 
This approach originated with the development of Higher-Order SVD (HOSVD) \cite{de2000multilinear}. 
The basic idea behind HOSVD is to identify a low-dimensional subspace within the column span of each unfolding of the tensor $\mathcal{X}$, denoted as $X_{(i)}$ for $ i=1,...,d$. 
Although HOSVD provides a sub-optimal solution, it is often used as an initialization for other methods when the computational cost is reasonable.

The authors who presented HOSVD proposed a method called Higher Order Orthogonal Iteration (HOOI) \cite{de2000best}. HOOI seeks orthonormal basis for the dominant subspace of each $Y_{(i)}$, which is the matricization of the tensor $\mathcal{Y}=\mathcal{X}\times_{-i} \{ U^T \}$ (see Definition \ref{TensorY}).
It is performed using a least-square approach while fixing other factor matrices.
By finding a low dimension subspace of $Y_{(i)}$ instead of $X_{(i)}$, HOOI provides a better low multi-linear rank$-(r_1, ... , r_d)$ approximation of $\mathcal{X}$ compared to HOSVD.
The Sequentially Truncated HOSVD (ST-HOSVD) is a variation of HOSVD that improves its efficiency. 
After finding each factor matrix, the tensor is projected using the obtained factor matrix, and the remaining operations are performed on the projected tensor. 
This method was introduced by Vannieuwenhoven et al. in 2012 \cite{vannieuwenhoven2012new}.
The HOOI method required careful initialization for efficient convergence, while the ST-HOSVD method performs better by choosing a suitable mode sequence - the order in which modes are processed.
Multi-linear Principal Component Analysis (MPCA) \cite{lu_plataniotis_venetsanopoulos_2008} is another method in this category that is similar to HOSVD but focuses on maximizing the variation in the projected tensor $\mathcal{C}$.
In the literature, various versions of HOSVD have been discussed, including hierarchical \cite{grasedyck2010hierarchical}, streaming \cite{sun2020low}, parallel \cite{austin2016parallel}, randomized \cite{che2019randomized}, and scalable \cite{oh2017s}. 
Recently, a fast and memory-efficient method called D-Tucker was introduced in \cite{jang2020d}.

The second approach to solving the Tucker decomposition problem involves using common second-order optimization algorithms. 
Eldén and Savas \cite{elden2009newton}, Savas and Lim \cite{savas2010quasi}, and Ishteva et al. \cite{ishteva2011best} have reformulated the original problem and apply Newton, quasi-Newton, and trust region methods, respectively, on the product of Grassmann manifolds. 
Although, utilizing the second-order information leads to algorithms with faster local convergence, these methods suffer from high computational complexity.

Tensor completion is a distinct problem from tensor decomposition, but noteworthy works include those of Kressner et al. \cite{kressner2014low} and Kasai et al. \cite{kasai2016low}. 
These studies are notable for using a first-order Riemannian method on a variant of tensor completion that employs Tucker decomposition.
In \cite{kressner2014low}, the Riemannian conjugate gradient method is applied to the manifold of tensors with fixed low multi-linear rank to solve the tensor completion problem. 
In \cite{kasai2016low}, the authors address the same problem by applying the same method as in \cite{kressner2014low} but on a product of Grassmann manifolds. 
The difference between our method and the latter is in the cost function and the optimization approach.

The structure of this paper is as follows: Section 2 presents some preliminary and background information. Section 3 discusses the problem description and reformulation, metric construction, and the proposed algorithms. In Section 4, the convergence analysis of the Riemannian coordinate descent algorithm is presented. Sections 5 and 6 contain the experimental results and conclusion, respectively.

\section{Preliminaries and background}
In this paper, calligraphic letters are used for representing tensors $(\mathcal{A}, \mathcal{B}, ...)$ and capital letters for representing matrices \linebreak $(A,B,...)$. In the following subsections, we provide some definitions and after that some background on the Riemannian preconditioning. 
\subsection{Definitions}
In this subsection, we give some definitions.
\begin{definition}[Tensor]
A tensor is a $d$-mode multi-dimensional array $\mathcal{X} \in \mathbb{R}^{n_1 \times \cdots \times n_d}$ with $n_i$ as the dimension of the $i$th mode. Each element in a tensor is denoted by $\mathcal{X}(k_1, ... , k_d),$ for $k_i \in [n_i]=\{1, ..., n_i\}$. Scalars, vectors and matrices are $0$-, $1$- and $2$-mode tensors, respectively.
\end{definition}

\begin{definition}[Matricization (unfolding)]
The matricization along the $i$th \linebreak mode, denoted by $X_{(i)} \in \mathbb{R}^{n_i \times \prod_{j\neq i}n_j}$, is constructed by putting tensor fibers of the $i$th mode alongside each other.
Tensor mode-$i$ fibers are determined by fixing indices in all modes except the $i$th mode, i.e. $\mathcal{X}(k_1,...,k_{i-1},:,k_{i+1},...,k_d)$.
\end{definition}

\begin{definition}[Multi-linear rank]
A tensor is called a rank-$(r_1, ... , r_d)$ tensor, if we have $rank(X_{(i)})=r_i$, for $i=1,...,d$, which indicates the dimension of the vector space spanned by mode-$i$ fibers. 
It is a generalization of the matrix rank.
\end{definition}

\begin{definition}[$i$-mode product] \label{i-mode}
For tensor $\mathcal{X} \in \mathbb{R}^{n_1 \times \cdots \times n_d}$ and matrix $A\in\mathbb{R}^{m\times n_i}$, the $i$-mode product $\mathcal{X}\times_i A \in \mathbb{R}^{n_1 \times \cdots \times n_{i-1} \times m \times n_{i+1} ... \times n_d}$ can be computed by the following formula:
\begin{equation*}
    (\mathcal{X}\times_i A)(k_1,...,k_{i-1},l,k_{i+1},...,k_d) = \sum_{k_i = 1}^{n_k} \mathcal{X}(k_1,...,k_i,...,k_d)A(l,k_i) \ .
\end{equation*}
	
Because of the relation $(\mathcal{X}\times_i A)_{(i)} = AX_{(i)}$, $i$-mode product can be thought of a transformation from a $n_i$-dimensional space to a $m$-dimensional space. 
\end{definition}

\begin{definition}[Tensor norm] The norm of a tensor $\mathcal{X}$ is given by
	\begin{displaymath}
		\lVert \mathcal{X} \rVert_F = \lVert X_{(i)} \rVert_F = \lVert vec(\mathcal{X}) \rVert,
	\end{displaymath}
	where $F$ is the Frobenious norm and $vec(.)$ is the vectorization operator.
\end{definition}

\begin{definition}[Stiefel manifold $St(n,r)$] \label{Stiefel}
    The set of all orthonormal $r_i$ frames in $\mathbb{R}^{n_i}$ is called the Stiefel manifold:
    \begin{equation*}
        St(n,r) = \{ X\in \mathbb{R}^{n\times r} : X^T X = I_r\}.
    \end{equation*}
\end{definition}

In this manifold, \emph{tangent vectors} at a point $X$ can be written as $\xi_X = X\Omega + X^\perp B$, where $\Omega \in Skew(r) = \{A\in \mathbb{R}^{r \times r} : A^T = -A\}$ and $X^\perp$ completes the orthonormal basis formed by $X$, so $X^T X^\perp = 0$. 
If vectors in the normal space are identified by $\nu_X = XA$, we can specify $A$ by implying the orthogonality between tangent vectors and normal vectors.
\begin{equation*}
    \xi_X \perp \nu_X \ : \ \langle \xi_X , \nu_X \rangle =  \langle X\Omega + X^\perp B, XA \rangle = 0 \quad \Longrightarrow \quad A \in Sym(r),
\end{equation*}
where $Sym(r)$ is the set of all $r \times r$ symmetric matrices.

The \emph{projection} of an arbitrary vector $Z \in \mathbb{R}^{n\times r}$ onto the tangent space is given by $Proj_X Z = Z - XA$, wherein $A$ is chosen such that the projection complies to the tangent vectors constraint, i.e. $\xi^T X + X^T \xi = 0$:
\begin{equation*}
    (Z - XA)^T X + X^T (Z-XA) = 0 \qquad \Longrightarrow \qquad A = sym(X^T Z),
\end{equation*}
where $sym(\cdot)$ returns the symmetirc part of the input matrix.

In a Stiefel manifold, like any embedded submanifold, the Riemannian gradient $\nabla f$ can be obtained by projecting the Euclidean gradient $G$ onto the tangent space of the current point.
\begin{equation*}
    \nabla f(X) = Proj_{X} G(X) = G(X) - X sym(X^T G(X)).
\end{equation*}

A \emph{retraction} $\mathcal{R}_x : T_x \mathcal{M} \rightarrow \mathcal{M}$ at the point $x$ on the manifold $ \mathcal{M}$ is a way of moving along a direction in the tangent space $T_x \mathcal{M}$ while staying on the manifold. More on that can be found in \cite[chapter 3.6]{boumal2020introduction}. 
For the Stiefel manifold we use $QR$-decomposition for retraction, that is $\mathcal{R}_X(\xi_X) = qr(X+\xi_X)$, where $qr$ is the orthonormal part in the $QR$-decomposition.	

\begin{definition}[Grassmann manifold $Gr(n,r)$]
We define two matrices $X$ and $Y$ to be equal under \emph{equivalence relation $\sim$} over $St(n,r)$, if their column spaces span the same subspace. We can define one of these matrices as a transformed version of the other, i.e., $X=YQ$, for some $Q\in O(r)$, where $O(r)$ is the set of all $r \times r$ orthogonal matrices.
	
We identify elements in the Grassmann manifold with this equivalence class, that is:
\begin{equation*}
    [X] = \{Y\in St(n,r) : X \sim Y \} = \{XQ : Q\in O(r)\}.
\end{equation*}
\end{definition}

The Grassmann manifold $Gr(n, r)$ is a quotient manifold, $St(n, r)/O(r) = \{[X]: X \in St(n,p)\}$, which represents the set of all linear $r$-dimensional subspaces in a $n$-dimensional vector space.
	
Consider a quotient manifold that is embedded in a total space $\mathcal{M}$
given by the set of equivalence relation $\sim$. 
A Riemannian metric $\langle \cdot, \cdot \rangle_x$ at $x\in \mathcal{M}$ in the total space can induce a Riemannian metric $\langle \cdot, \cdot \rangle_{[x]}$ on the quotient manifold $\mathcal{M}/\sim$
\begin{equation*}
    \langle \xi_{[x]}, \eta_{[x]} \rangle_{[x]} = \langle \xi_x, \eta_x\rangle_x,
\end{equation*}
where vectors $\xi_x$ and $\eta_x$ belong to $\mathcal{H}_x$, the horizontal space of $T_x \mathcal{M}$.
This subspace provides a valid matrix representation of the abstract tangent space $T_{[x]}\mathcal{M}/\sim$ as $\xi_x$ and $\eta_x$ are unique representations of the abstract tangent vectors $\xi_{[x]}$ and $\eta_{[x]}$, respectively.
The horizontal space is the orthogonal complement to the vertical space in this Riemannian metric.
The vertical space is defined as $\mathcal{V}_x = \operatorname{ker}D\pi(x)$, where $\pi: x \mapsto \pi(x)=[x]$ is the natural projection that links the total space to its quotient.

If the cost function in the total space does not change in the directions of vectors in the vertical space, then the Riemannian gradient in the quotient manifold is given by,
\begin{equation*}
    \nabla_{[x]}f = \nabla_xf.
\end{equation*}

A \emph{retraction} operator $\mathcal{R}_x : \mathcal{H}_x \rightarrow \mathcal{M}$ can be given by,
\begin{equation*}
    \mathcal{R}_{[x]}(\xi_{[x]}) = [\mathcal{R}_x(\xi_x)],
\end{equation*}
where $\mathcal{R}_x(.)$ is a retraction in the total manifold.
For further information on the aforementioned concepts in Riemannian manifold optimization, refer to \cite{absil2009optimization, boumal2020introduction}.

\subsection{Riemannian preconditioning}
Mishra and Sepulchre \cite{mishra2016riemannian} brought attention to relation between sequential quadratic programming which embeds constraints into the Lagrangian and the Riemannian Newton method which encodes constraints into the search space. 
Then, they exploited this relation and introduced a way of building Riemannian metrics. Here we bring the gist of their work.

Consider the optimization problem,
\begin{equation} \label{equality-const}
	\begin{split}
		\min_{x\in \mathbb{R}^n} \quad & f(x), \\
		\textrm{s.t.} \quad & h(x) = 0,
	\end{split}
\end{equation}
where $f: \mathbb{R}^n \rightarrow \mathbb{R}$ and $h: \mathbb{R}^n \rightarrow \mathbb{R}^p, n\geq p$ are smooth functions. Sequential quadratic programming deals with the the unconstrained Lagrangian  which is defined as
\begin{equation*}
	\mathcal{L}(x,\lambda) = f(x) - \langle \lambda, h(x) \rangle,
\end{equation*}
in which $\lambda \in \mathbb{R}^p$ represents the Lagrange multiplier. From the optimality condition we know that at a local minimum $x$, Lagrange multiplier $\lambda_x$, can be obtained by 
\begin{equation*}
	h_x(x) \lambda_x = f_x(x),
\end{equation*}
where $f_x$ is the first-order derivative of the cost function $f(x)$ and $h_x(x)\in \mathbb{R}^{n\times p}$ is the Jacobian of the constraints  $h(x)$. Applying the pseudo inverse of $h_x(x)$ to $f_x(x)$, we arrive at 
\begin{equation*}
	\lambda_x = (h_x(x)^T h_x(x))^{-1} h_x(x)^T f_x(x),
\end{equation*}
where $h_x(x)$ is assumed to have a full column rank.
Therefore the set $\mathcal{M}:=\{ x \in \mathbb{R}^n \ | \ h(x)=0 \}$ has an embedded differentiable submanifold structure, and we can recast the equality constrained problem \cref{equality-const} as an unconstrained optimization problem on a nonlinear search space.
The authors stated that in the neighborhood of a local minimum, the second-order derivative of the Lagrangian in the total space efficiently gives us the second-order information of the problem. 
The theorem below is brought for more clarification.

\begin{theorem}[Theorem 3.1 in \cite{mishra2016riemannian}] \label{2.9}
	Consider an equivalence relation $\sim$ in $\mathcal{M}$. Assume that both $\mathcal{M}$ and $\mathcal{M}/$$\sim$ have the structure of a Riemannian manifold and a function $f: \mathcal{M}\rightarrow \mathbb{R}$ is a smooth function with isolated minima on the quotient manifold. Assume also that $\mathcal{M}$ has the structure of an embedded submanifold in $\mathbb{R}^n$.
	If $x^* \in \mathcal{M}$ is a local minimum of $f$ on $\mathcal{M}$, then the following hold:
	\begin{itemize}
		\item $
		\langle \eta_{x^*} , D^2 \mathcal{L}({x^*}, \lambda_{x^*}) [\eta_{x^*}] \rangle = 0, \quad \forall \eta_{x^*} \in \mathcal{V}_{x^*},
		$
		\item the quantity $\langle \xi_{x^*} , D^2 \mathcal{L}({x^*}, \lambda_{x^*}) [\xi_{x^*}] \rangle$ captures all second-order information of the cost function $f$ on $\mathcal{M}/\sim$ for all $\xi_{x^*} \in \mathcal{H}_{x^*}$,
	\end{itemize}
	where $\mathcal{V}_{x^*}$ is the vertical space, and  $\mathcal{H}_{x^*}$ is the horizontal space (that subspace of $T_{x^*}\mathcal{M}$ which is orthogonal to the vertical space) and $D^2 \mathcal{L}({x^*}, \lambda_{x^*}) [\xi_{x^*}]$ is the second-order derivative of $\mathcal{L}({x}, \lambda_{x})$ with respect to $x$ at $x^*\in \mathcal{M}$ applied in the direction of $\xi_{x^*} \in \mathcal{H}_{x^*}$ and keeping $\lambda_{x^*}$ fixed to its least-squares estimate.
\end{theorem}

The proper search direction in  sequential quadratic programming  is computed by solving the following optimization problem in the neighborhood of a minimum:
\begin{displaymath}
	\arg\min_{\xi_x \in \mathcal{H}_x} f(x) - \langle f_x(x), \xi_x \rangle + \frac{1}{2} \langle \xi_{x} , D^2 \mathcal{L}({x}, \lambda_{x}) [\xi_{x}] \rangle .
\end{displaymath}

If $\langle \xi_x , D^2 \mathcal{L}(x_k, \lambda_x) [\xi_x] \rangle$ is strictly positive for all tangent vectors $\xi_x$ in the horizontal space $\mathcal{H}_x$ at the point $x$, then this optimization problem has a unique solution. After updating the variables by moving along the obtained direction, to maintain strict feasibility, it needs a projection onto the constraint, thus they name this method \emph{feasibly projected sequential quadratic programming}.

Now that we know that the Lagrangian captures second-order information of the problem, the authors in \cite{mishra2016riemannian} introduced a family of regularized metrics that incorporate the second information by using the second-order derivative of the Lagrangian,
\begin{displaymath}
	\langle\xi_x, \eta_x\rangle_x =\omega_1 \langle \xi_x, D^2 f(x)[\eta_x] \rangle + \omega_2 \langle \xi_x, D^2 c(x, \lambda_x)[\eta_x] \rangle ,
\end{displaymath}
in which $c(x, \lambda_x) = - \langle \lambda_x, h(x) \rangle$ and $\omega_1 \in [0,1], \omega_2 \in [0,1]$. The first and second terms of this regulated metric correspond to the cost function and the constraint, respectively. In addition to invariance, the metric needs to be positive definite, so:
\begin{align*}
	&\text{if} \ \ D^2f \succ 0 \quad \text{then} \quad \omega_1 = 1,\quad \omega_2 = \omega \in [0,1), \\
	&\text{if} \ \ D^2f \prec 0 \quad \text{then} \quad \omega_2 = 1,\quad \omega_1 = \omega \in [0,1),
\end{align*}
where $\omega$ can also be updated in each iteration by a rule like $\omega^k = 1 - 2^{1-k}$. 
We refer the reader to the original paper for further details \cite[Section 3.3]{mishra2016riemannian}.
Mishra and Kasai in \cite{kasai2016low} exploited the idea of Riemannian preconditioning for tensor completion.

\section{Problem statement}
In the Tucker decomposition, we want to decompose a dense $d$-mode tensor $\mathcal{X} \in \mathbb{R}^{n_1 \times \cdots \times n_d}$ into a core tensor $\mathcal{C} \in \mathbb{R}^{r_1 \times \cdots \times r_d}$ and $d$ orthonormal factor matrices $U_i \in St(n_i,r_i)$. 
The Domain of the objective function is the following product manifold,
\begin{equation*}
	(\mathcal{C}, U_1, ..., U_d)\in \mathcal{M} := \mathbb{R}^{r_1\times \cdots \times r_d} \times St(n_1,r_1) \times \cdots \times St(n_d,r_d).
\end{equation*}

We solve the following optimization problem:
\begin{equation}
    \label{TD}
    \min_{(\mathcal{C},U_1,...,U_d)\in\mathcal{M}} \quad \lVert \mathcal{X} - \mathcal{C} \times_1 U_1 \times_2 \cdots \times_d U_d \rVert_{F}^{2},
\end{equation}
where $\lVert . \rVert_F$ is the Frobenius norm. 
The objective function has a symmetry for the manifold of orthogonal matrices $O(r_i)$, i.e.,
\begin{equation*}
    f(\mathcal{C},U_1,...,U_d) = f(\mathcal{C}\times_1 O_1^T \times_2 \cdots \times_d O_d^T, U_1 O_1, \cdots, U_d O_d).
\end{equation*}
So, this problem is actually an optimization problem on the product of Grassmann manifolds.

We know from \cite{de2000best} that 
\begin{equation}
    \label{eq:Error}
    \lVert \mathcal{X} - \mathcal{C}\times \{U\} \rVert^2_F = \lVert \mathcal{X} \rVert^2_F -\lVert \mathcal{C} \rVert^2_F,
\end{equation}
where $\mathcal{C}\times\{U\} = \mathcal{C}\times_1 U_1 \times_2 \cdots \times_d U_d $. This means that the minimization of the \emph{reconstruction error} is equivalent to the maximization of \emph{the energy of the projected tensor}. So, for solving the problem \cref{TD} in a coordinate descent fashion we can recast it as a series of subproblems involving the following maximization problem for the $ith$ factor matrix:
\begin{equation} \label{TD reformulate}
	\max_{[U_i] \in Gr(n_i, r_i)} \frac{1}{2} \lVert U_i^T Y_{(i)} \rVert^2_F. 
\end{equation}
In this problem $Y_{(i)}$ is the matricization of the tensor $\mathcal{Y}_i$ along the $i$th mode. The tensor $\mathcal{Y}_i \in \mathbb{R}^{r_1\times \cdots \times r_{i-1} \times n_i \times r_{i+1} \times \cdots \times r_d}$ is produced by projecting the tensor $\mathcal{X}$ to a lower dimensional subspace by the help of the assumed fixed factor matrices while mode $i$ is excluded, i.e.,
\begin{equation}
    \label{TensorY}
    \mathcal{Y}_{i} = \mathcal{X}\times_{-i} \{ U^T \} =  \mathcal{X}\times_1 U_1^T \times_2 \cdots \times_{i-1} U_{i-1}^T \times_{i+1} U_{i+1}^T \times_{i+2} \cdots \times_d U_d^T. 
\end{equation}

If we proceed to decompose a dense tensor $\mathcal{X}$ with a low multi-linear rank using the formulation given in \cref{TD reformulate} on the product of Grassmannian manifolds equipped with the Euclidean metric
\begin{displaymath}
    \langle\xi_{U_i}, \eta_{U_i}\rangle_{U_i} = \operatorname{trace}(\xi_{U_i}^T \eta_{U_i}),
\end{displaymath}
we observe that coordinate descent would have poor convergence results. The relative error plots for $10$ random samples of $\mathcal{X}\in \mathbb{R}^{100\times 100\times 100}$ with multi-linear rank-$(5,5,5)$ can be seen in \cref{metric}.

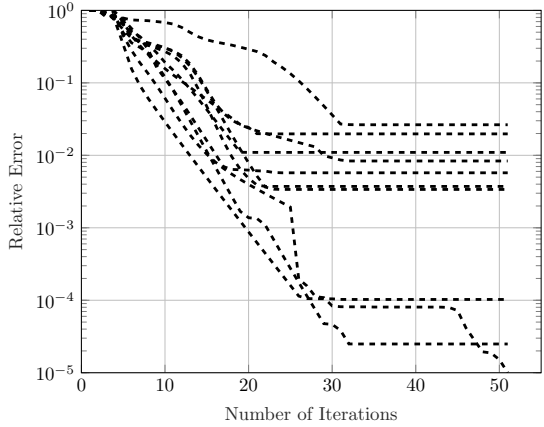
\begin{figure}[tb] \label{metric}
\centering
  \resizebox{0.55\textwidth}{!}{\input{tikzfig/euclidean_norm.tex}}
  \caption{Convergence of Riemannian coordinate descent with the Euclidean metric for decomposing a random tensor having a low multi-linear rank. The best attainable relative error is zero, and it is clear that the coordinate descent method has convergence problems in the Euclidean space.}
 \end{figure}
To give a remedy for the slow convergence using the Euclidean metric, in the next subsection we apply Riemannian preconditioning to construct a new Riemannian metric which we will see in the experiments that it results in an excellent experimental performance.

\subsection{Riemannian Preconditioned Coordinate Descent}
In this section, we want to utilize the idea of Riemannian preconditioning in solving the problem \cref{TD reformulate}. For the problem
\begin{equation*}
	\max_{[U_i] \in Gr(n_i, r_i)} \frac{1}{2} \lVert U_i^T Y_{(i)} \rVert_F^2 ,
\end{equation*}
where $Y_{(i)} \in \mathbb{R}^{n_i \times \prod_{j=1, j\neq i}^{d} r_j}$, the Lagrangian is defined as
\begin{equation*}
	\mathcal{L}(U_i, \lambda) = -\frac{1}{2} Trace(Y_{(i)}^TU_iU_i^TY_{(i)}) +\frac{1}{2}\langle\lambda, U_i^TU_i - I\rangle ,
\end{equation*}
in which $\lambda \in \mathbb{R}^{r \times r}$ is a symmetric matrix representing the Lagrange multiplier. The gradient of the Lagrangian w.r.t. $U_i$ is
\begin{displaymath}
	\mathcal{L}_{U_i}(U_i, \lambda) = -Y_{(i)}Y_{(i)}^TU_i + U_i\lambda,
\end{displaymath}
which due to the optimality condition must be equal to zero at a local minimum.
We estimate $\lambda$ in a least square sense as follows:
\begin{displaymath}
	\lambda_{U_i} = U_i^TY_{(i)}Y_{(i)}^TU_i.
\end{displaymath}

We assumed that $\lambda_{U_i}$ would be invertible. 
It is because we know for a fact that  $r_i \ll n_i$ and as a result $r_i < \prod_{j\neq i} r_j$, so it is a reasonable guess that $\lambda_{U_i}$ would be invertible \cite{carlini2011ranks}.
On top of that, in all of the experiments we have conducted, we never had any issue with the invertiblity of $\lambda_{U_i}$.

The second-order derivative of the Lagrangian along $\xi_{U_i}$ while keeping $\lambda_{U_i}$ fixed (see \cref{2.9}), is
\begin{displaymath}
	D^2 \mathcal{L}(U_i, \lambda_{U_i}) [\xi_{U_i}] = -Y_{(i)}Y_{(i)}^T\xi_{U_i} + \xi_{U_i} \lambda_{U_i} .
\end{displaymath}
As we explained in Section 2.2, the second-order derivative of the Lagrangian captures all the second-order information of the cost function near the minimum. So, to incorporate this information we introduce the Riemannian metric,
\begin{displaymath}
\langle\eta_{U_i}, \xi_{U_i}\rangle_{U_i} = -\omega_1\langle\eta_{U_i}, Y_{(i)}Y_{(i)}^T\xi_{U_i}\rangle +  \omega_2\langle \eta_{U_i}, \xi_{U_i} \lambda_{U_i} \rangle ,
\end{displaymath}
where $\eta_{U_i}, \xi_{U_i} \in \mathbb{R}^{n_i \times r_i}$ are tangent vectors in $T_{U_i} \mathcal{M}$. If we could have set both $\omega_1$ and $\omega_2$ equal to one, then the Riemannian gradient using this metric would have been the Euclidean Newton direction. 
But the constants $\omega_1 \in [0, 1]$ and $\omega_2 \in [0, 1]$ should be chosen in a way to make sure that the proposed metric stays positive definite for all points on the manifold. 
Since, the matrix $-Y_{(i)}Y_{(i)}^T$ is negative semi-definite, we choose $\omega_1 = 0$ and $\omega_2=1$, which results in
\begin{equation} \label{rim-metric}
    \langle\eta_{U_i}, \xi_{U_i}\rangle_{U_i} = \langle \eta_{U_i}, \xi_{U_i} \lambda_{U_i} \rangle.
\end{equation}
Variables in the search space are invariant under the symmetry transformation, therefore the proposed metric must be invariant under the associated symmetries, i.e.
\begin{displaymath}
	U_i \mapsto U_iQ \quad \text{and} \quad \lambda_{U_i} \mapsto Q^T\lambda_{U_i} Q , \quad Q\in O(r).
\end{displaymath}
It can be verified that this property holds for the Riemannian metric \cref{rim-metric}:
\begin{displaymath}
    \langle \eta_{U_{i}} Q, \xi_{U_{i}} Q \rangle_{U_{i}} = \langle \eta_{U_{i}} Q, \xi_{U_{i}} Q Q^T \lambda_{U_{i}} Q \rangle = \langle \eta_{U_{i}}, \xi_{U_{i}} \lambda_{U_{i}} \rangle = \langle \eta_{U_{i}}, \xi_{U_{i}} \rangle_{U_{i}}.
\end{displaymath}

In embedded submanifolds, the Riemannian gradient is obtained by orthogonally projecting the Euclidean gradient onto the tangent space.
Tangent vectors at the point $U_i$ in a Stiefel manifold can be represented by $\xi = U_i\Omega + U_i^{\perp}B \in T_{U_i} \mathcal{M}$, where $\Omega \in Skew(r)$.
Writing normal vectors as $\nu = U_iA + U_i^\perp C \in N_{U_i} \mathcal{M}$, we have
\begin{align*}
    0 = \langle\xi, \nu\rangle_{U_{i}}&=\left\langle U_{i} \Omega+U_{i}^{\perp} B, U_{i} A + U_i^\perp C \right\rangle_{U_{i}} \\
    &= \left\langle U_{i} \Omega+U_{i}^{\perp} B, U_{i} A \lambda_{U_i} + U_i^\perp C \lambda_{U_i}\right\rangle \\
    &= \left\langle\Omega, A \lambda_{U_{i}}\right\rangle + \langle B, C \lambda_{U_i}\rangle.
\end{align*}
Considering $\lambda_{U_i}$ to be invertible, it results in $C=0$ and
\begin{displaymath}
	A = S \lambda_{U_i}^{-1}   , \quad S\in Sym(r).
\end{displaymath}
Therefore, by putting the normal vectors at $U_i$ as $\nu = U_i S \lambda_{U_i}^{-1} $, the projection of a given matrix $G$ onto the tangent space of the Stiefel manifold is given by,
\begin{displaymath}
\operatorname{Proj}_{U_i} G = \ G - U_i S \lambda_{U_i}^{-1},
\end{displaymath}
which must comply to the tangent vector constraint on the manifold
\begin{displaymath}
	U_i^T (\operatorname{Proj}_{U_i} G) + (\operatorname{Proj}_{U_i}G)^T U_i = 0 .
\end{displaymath}
Consequently, we have the following Sylvester equation for the symmetric matrix $S$:
\begin{displaymath}
\lambda_{U_i}S + S\lambda_{U_i} \ = \ \lambda_{U_i} (U_i^T G + G^T U_i)\lambda_{U_i}.
\end{displaymath}
By Riemannian submersion theory \cite[section 3.6.2]{absil2009optimization} , we know that this projection belongs to the horizontal space. Thus, there is no need for further projection onto the horizontal space. If we define $G$ as the Euclidean gradient in the total space, we can simply compute the Riemannian gradient by
\begin{equation*}
	\begin{aligned}
		\nabla f_{[U_i]} =& \ G + U_i .
	\end{aligned}
\end{equation*}

Although we used second-order information to form a new preconditioned metric, when we apply the Riemannian coordinate descent method to the Tucker decomposition problem, it leads to a simple first-order algorithm. 
The preconditioned metric gives us the second-order insight into the problem but the method itself is a first-order method.

With the help of the metric, we introduce the proposed RPCD method in \cref{alg:RPCD}. RPCD is memory efficient, which is desirable because we wanted to reduce the storage complexity of the original tensor $\mathcal{X}$ in the first place. 
To be specific, assume $n_i = n$ and $r_i = r$, then tensor $\mathcal{Y}_i$ has $nr^{d-1}$ elements and the Euclidean and the Riemannian gradient both have $nr$ elements. 
Empirically, we observed that the best choice for the step size is $\alpha = 1$.
In this case, one step of the inner loop in RPCD is equivalent to one step of the classic orthogonal iteration method (also called subspace iteration) \cite[Section 8.2.4]{golub1996matrix} for finding invariant subspaces.
In other words, we have shown that the orthogonal iteration method can be seen as a preconditioned Riemannian gradient descent algorithm. 
The unit step size is also compatible with the fact that the obtained direction is an approximation of the Newton direction.

\begin{algorithm}[tbhp]
\caption{RPCD/RPCD+}
\label{alg:RPCD}
\begin{algorithmic}
    \Require Dense tensor $\mathcal{X}$, set of orthonormal factor matrices $\{U\}$ randomly initialized on the Stiefel manifold, retraction operator $\mathcal{R}$, step size $\alpha$,  Plus\_flag for selecting between RPCD and RPCD+, tolerance error of the outer loop $\epsilon$, tolerance error of the inner loop specific for RPCD+ $\epsilon'$,  maximum number of iterations in the outer loop Maxiter and maximum number of iterations in the inner loop specific for RPCD+ Maxiter\_inner.
    \For{$k=1:\text{Maxiter}$}
        \For{$i=1:d$}
            \State $\mathcal{Y}_i \leftarrow \mathcal{X} \times_{-i}\{U^T\}$
            \State $U_i , C_i \leftarrow$ UPDATE($\mathcal{R}_{U_i}, Y_{(i)}, U_i$)
            \If{Plus\_flag = true}
            \State $\bar{E}_{0} \leftarrow  \sqrt{\lVert \mathcal{X} \rVert_F^2 -\lVert C_i \lVert_F^2 }/ \lVert \mathcal{X} \rVert_F$
                \For{$k'=1:\text{Maxiter\_inner}$}
                \State $U_i, C_i  \leftarrow$ UPDATE($\mathcal{R}_{U_i}, Y_{(i)}, U_i$)
                \State $\bar{E}_{k'} \leftarrow  \sqrt{\lVert \mathcal{X} \rVert_F^2 -\lVert C_i \lVert_F^2 }/ \lVert \mathcal{X} \rVert_F$
                    \If{$\bar{E}_{k'} - \bar{E}_{k'-1} \leq \epsilon'$ }
                        \State break
                    \EndIf
                \EndFor
            \EndIf
        \EndFor
        \State $E_{k} \leftarrow  \sqrt{\lVert \mathcal{X} \rVert_F^2 -\lVert U_d^T Y_{(d)} \lVert_F^2 }/ \lVert \mathcal{X} \rVert_F$
         \If{$k \geq 2$}
        \If{$E_k - E_{k-1} \leq \epsilon$}
            \State break 
        \EndIf
         \EndIf
    \EndFor
    \Function{UPDATE}{$\mathcal{R}_{U_i}, Y_{(i)}, U_i$} \label{alg:update}
    	\State $C_i  \leftarrow U_i^TY_{(i)}$
        \State $G \leftarrow -Y_{(i)}C_i^T$
        \State $\nabla f \leftarrow G + U_i $
        \State $U_i \leftarrow \mathcal{R}_{U_i}(U_{i} - \alpha \nabla f)$
        \State \Return $U_i, C_i $
    \EndFunction
    \Ensure Set of factor matrices $\{U\}$
\end{algorithmic}
\end{algorithm}

In \cref{alg:RPCD} for the retraction, we use the \emph{QR-decomposition} implemented by the \emph{Householder} algorithm in Matlab which has computational complexity $\mathcal{O}(nr^2)$ for $n \times r$ matrices.
For the stopping criterion, we use relative error delta which is the amount of difference in the relative error in two consecutive iterations, i.e., $|\text{E}_k - \text{E}_{k-1}|<\epsilon$, 
\begin{equation}
    \text{E} = \frac{\| \mathcal{X} - \hat{\mathcal{X}} \|_F}{\| \mathcal{X} \|_F} = \frac{\sqrt{\lVert \mathcal{X} \rVert_F^2 -\lVert U_i^T Y_{(i)} \lVert_F^2}}{\lVert \mathcal{X} \rVert_F}
\end{equation}
where second equality comes from equation \eqref{eq:Error} and $\Vert \mathcal{C} \rVert_F = \lVert U_i^T Y_{(i)} \lVert_F$.
Since computing the tensor $\mathcal{Y}_i$ is a lot more expensive than the rest of the inner loop computations, $\mathcal{O}(n^dr^{d-1})$, so it would be a good idea to do multiple updates in every inner loop.
In RPCD+, when \texttt{plus\_flag = true}, we repeat the updating process as long as the change in the relative error would be less than a certain threshold $\epsilon^\prime$, which can be much smaller than the stopping criterion threshold $\epsilon$.

In the next section, we provide a convergence analysis for the proposed method as an extension of the coordinate descent method to the Riemannian domain in a special case that the search space is a product manifold.

\section{Convergence analysis}
The RPCD method can be thought of as an extension of Tangent Subspace Descent (TSD) \cite{gutman2022coordinate}. 
TSD is a recent generalization of the coordinate descent method to the manifold domain. 
To provide a convergence analysis for RPCD, we generalize the convergence analysis of \cite{gutman2022coordinate} to the case of product manifolds where the exponential map and parallel transport are substituted by  retraction and  vector transport, respectively. 
Convergence analysis of the TSD method is a generalization of the Euclidean block coordinate descent method described in \cite{beck2013convergence}. 
The TSD method with retraction and vector transport is outlined in \cref{alg:TSD}. 
The projections in TSD are updated in each iteration of the inner loop with the help of the vector transport operator.

\begin{algorithm}[tbhp]
	\caption{TSD with retraction and vector transport}
	\label{alg:TSD}
	\begin{algorithmic}
		\State Given $\mathcal{R}_x(\xi)$ as a retraction from a point $x$ in the direction of $\xi$ and $\mathcal{T}_x^y$ as a vector transport operator  from a point $x$ to a point $y$, see \cref{vec-trans}.
		\Require Initial point $x^{0}\in \mathcal{M}$, and $\tilde{P}^0 = \{P_i^{x^0}\}_{i=1}^m$ are orthogonal projections onto $m$ orthogonal subspaces of the tangent space at $x^{0}$
		\For{$t = 1, 2, ...$}
		\State Set $y^0 := x^{t-1}, \ \tilde{P}^{y^0} := \tilde{P}^{t-1}$
		\For{$k = 1, ... , m$}
		\State $\alpha_k=\frac{1}{L_k}$
		\Comment{{\tiny $L_k$ is the Lipschitz constant for each block of variables determined by \cref{lip-block}}}
		\State Update $y^{k}=\mathcal{R}_{y^{k-1}}(-\alpha_k P^{y^{k-1}}_{k}\nabla f(y^{k-1}))$
		\State Update $P^{y^{k}}_{i}=\mathcal{T}^{y^{k}}_{y^{k-1}}P^{y^{k-1}}_{i}\mathcal{T}^{y^{k-1}}_{y^{k}}$ for $i = 1, ... , m$.
		\EndFor
		\State Update $x^t := y^m, \ \tilde{P}^{t} := \tilde{P}^{y^m}$
		\EndFor
		\Ensure Sequence $\{x^t\} \subset \mathcal{M}$
	\end{algorithmic}
\end{algorithm}
Before, we start to study the convergence analysis, it would be helpful to quickly review some definitions:

\begin{definition}[Vector transport {\cite[Definition 8.1.1]{absil2009optimization}}] \label{vec-trans}
A vector transport on a manifold $\mathcal{M}$ is a smooth mapping
\begin{displaymath}
    \mathcal{T}^{y^{k}}_{y^{k-1}} : T_{y^{k-1}}\mathcal{M} \mapsto T_{y^{k}}\mathcal{M},
\end{displaymath}
associated with a retraction $y^k = \mathcal{R}_{y^{k-1}}(\eta)$.
\end{definition}
Here we assume that our vector transport is an isometry.
See \cite[Section 10.5]{boumal2020introduction} for other properties.

\begin{definition}[Radially Lipschitz continuously differentiable function]
We say that the pull-back function $f \circ \mathcal{R}$ is radially Lipschitz continuously differentiable for all $x \in \mathcal{M}$ if there exist a positive constant $L_{RL}$ such that for all $x$ and all $\xi \in T_x\mathcal{M}$ the following holds for $r>0$,
\begin{equation*}
   \Big | \frac{d}{dt}(f\circ\mathcal{R})(t\xi)|_{t=r} - \frac{d}{dt}(f\circ\mathcal{R}) (t\xi)|_{t=0} \Big | \leq r L_{RL} \|\xi \|
\end{equation*}
where $\frac{d}{dt}$ is the first-order derivative of a single-variable function.
\end{definition}

\begin{definition}[Operator $S^k$] \label{S}
	It is given as,
	\begin{displaymath}
		S^0=id_{T_{y^0}\mathcal{M}} \quad , \quad S^k=\mathcal{T}^{y^0}_{y^1}\cdots\mathcal{T}^{y^{k-1}}_{y^{k}}= S^{k-1} \mathcal{T}^{y^{k-1}}_{y^{k}} \ ; \ 1\le k \le l,
	\end{displaymath}
	where $id$ is the identity operator. With this operator, we can write the update rule for the projection matrices as $P^{y^{k}}_{i} = (S^k)^{-1} P^{y^0}_{i} S^k$.
\end{definition}

\begin{definition}[retraction-convex]
    Function $f:\mathcal{M} \rightarrow \mathbb{R}$ is retraction-convex w.r.t $\mathcal{R}$ for all $\eta \in T_x \mathcal{M}$, $ \left\| \eta \right\|_x=1$, if the pull-back function $f \circ \mathcal{R}_x$ is convex in its domain.
\end{definition}

\begin{proposition}[First-order characteristic of a retraction-convex function]
    If $f:\mathcal{M} \rightarrow \mathbb{R}$ is retraction-convex w.r.t a retraction $\mathcal{R}_x:T_x\mathcal{M \rightarrow \mathcal{M}}$, then we know  by definition that the pull-back function is convex. 
    For any direction $\eta \in T_x\mathcal{M}$, by the first-order characteristic of the convex function $f(\mathcal{R}_x(r\eta)):\mathbb{R}\rightarrow \mathcal{M}$, we have
    \begin{equation*}
       f(\mathcal{R}_x(r\eta)) \geq  f(\mathcal{R}_x(s\eta)) + (r-s)\frac{d}{dt}(f \circ \mathcal{R}_x)(t)|_{t=s},
    \end{equation*}
 where $r,s \in \mathbb{R}$.
    The second term can be interpreted as
    \begin{equation*}
        \frac{d}{dt}(f \circ \mathcal{R}_x)(t)|_{t=s} = Df(\mathcal{R}_x(s\eta))[\textbf{J}\mathcal{R}_x(s\eta)] = \langle \nabla f(\mathcal{R}_x(s\eta)), \textbf{J}\mathcal{R}_x(s\eta)\rangle_{\mathcal{R}_x(s\eta)},
    \end{equation*}
   where $\textbf{J}\mathcal{R}$ is the Jacobian of the retraction operator and $Df(\cdot)[\cdot]$ is the directional derivative of function $f$ in a specified direction.
    Thus, for $r=1$ and $s=0$,
    \begin{displaymath}
        f(\mathcal{R}_x(\eta)) \geq  f(x) + \langle \nabla f(x), \eta \rangle_x .
    \end{displaymath}
\end{proposition}

\begin{proposition}[Restricted Lipschitz-type gradient for the pullback function] \label{lip-grad}
We know by \cite[Lemma 2.7]{boumal2019global} that if $\mathcal{M}$ is a compact submanifold of Euclidean space and if $f$ has Lipschitz continuous gradients, then
\begin{displaymath}
    f(\mathcal{R}_x(\eta)) \leq  f(x) + \langle \nabla f(x), \eta \rangle_x + \frac{L_g}{2} \|\eta\|^2_x ,\qquad \forall \eta \in  T_x \mathcal{M},
\end{displaymath}
for some $L_g>0$.
\end{proposition}		

First, we study the first-order optimality condition in the following proposition. 

\begin{proposition}[Optimality condition]
	Assume $f$ is a retraction-convex function and there is a retraction curve between any two points on the Riemannian manifold $\mathcal{M}$, then
	\begin{displaymath}
		\nabla f(x^*)=0 \quad \Leftrightarrow \quad  x^* \ is \; a \; minimizer.
	\end{displaymath}
\end{proposition}
\begin{proof}
	Considering any differentiable curve $\mathcal{R}_{x^*}(t\eta)$, which starts at a local optimum point $x^*$, the pull-back function $f(\mathcal{R}_{x^*}(t\eta))$ has a minimum at $t=0$ because $\mathcal{R}_{x^*}(t\eta)\big|_{t=0} = x^*$. We know that $(f\circ \mathcal{R})'(0) = \langle \nabla f(x^*), \eta \rangle_{x^*}$, so for this to be zero for all $\eta \in T_{x^*}\mathcal{M}$, we must have $\nabla f(x^*) = 0$.
	
	From the first-order characteristic of the retraction-convex function $f$ we have,
	\begin{equation*}
		f(\mathcal{R}_{x^*}(\eta)) \geq  f(x^*) + \langle \nabla f(x^*), \eta \rangle_{x^*}, \qquad \forall \eta \in \mathcal{R}^{-1}_{x^*}(x),
	\end{equation*}
	and if $\nabla f(x^*) = 0$ then $f(x) \geq f(x^*)$, hence the point $x^*$ is a global minimum point.
\end{proof}

With the following Lip-Block lemma and the descent direction advocated by  \cref{alg:TSD}, we can prove the Sufficient Decrease lemma.

\begin{lemma}[Lip-Block] \label{lip-block}
    If $f$ has the restricted Lipschitz-type gradient, then for any $i,k \in \{1, ... , m\}$ and all $\nu \in Im(P^{k-1}_i) \subset T_{y^{k-1}}\mathcal{M}$, where $Im(\cdot)$ is the subspace that a projection matrix spans, there exist constants $0 < L_1, ... , L_m < \infty$ such that
    \begin{equation} \label{Lip-Block}
        f(\mathcal{R}_{y^{k-1}}(\nu)) \leq f(y^{k-1}) + \langle \nabla f(y^{k-1}), \nu \rangle_{y^{k-1}}  +\frac{L_i}{2}\lVert\nu\rVert^2 _{y^{k-1}}.
    \end{equation}
\end{lemma}
\begin{proof}
	By the fact that $\nu \in T_{y^{k-1}}\mathcal{M}$, it can be seen easily that \cref{Lip-Block} is the block version of the restricted Lipschitz-type gradient for the pullback function.
\end{proof}

\begin{lemma}[Sufficient decrease] \label{sufficient lemma}
    Assume $f$ has the restricted Lipschitz-type gradient, and furthermore $f \circ \mathcal{R}$ is a radially Lipschitz continuous differentiable function. 
	Using the projected gradient onto the $k$th subspace in each inner loop iteration of \cref{alg:TSD}, i.e., $\nu = - \frac{1}{L_k} P^{y^{k-1}}_k \nabla f(y^{k-1})$, where $P^{y^{k-1}}_k$ is the orthogonal projection to $k$th subspace of $T_{y^{k-1}}\mathcal{M}$, then we have
	\begin{equation} \label{Dec-inner}
		f(y^0) - f(y^m) \geq \sum_{k=1}^m\frac{1}{2L_k}\lVert P^{y^{k-1}}_k \nabla f(y^{k-1}) \rVert^2_{y^{k-1}} \ .
	\end{equation}
	The following inequality also holds
	\begin{equation} \label{Inequality} 
		\| P^{y^0}_i\nabla f(y^0) - P^{y^0}_{i}S^{i-1} \nabla f(y^{i-1}) \|_{y^0}^2\le C\sum_{k=1}^{i-1}\| P^{y^k-1}_k\nabla f(y^{k-1})\|_{y^{k-1}}^2 \ ,
	\end{equation}
	for $C=(m-1)L_{RL}^2/L_{min}^2$, where $L_{\min}=\min\{L_1, ..., L_m\}$ and $L_{RL}$ is the radially Lipschitz constant. 
	For more details on the operator $S$ see \cref{S}.
	Furthermore, there is a lower bound on the cost function decrease at each iteration of the outer loop in \cref{alg:TSD}:
	\begin{equation} \label{Dec-outer}
		f(y^0) - f(y^m) \ge \frac{1}{4L_{max}(1+Cm)}\|\nabla f(y^0)\|_{y^{0}}^2,
	\end{equation}
	where $L_{\max}=\max\{L_1, ..., L_m\}$.
\end{lemma}

\begin{proof}
The inequality \cref{Inequality} was proved in \cite[Lemma 4.3]{gutman2022coordinate} with a slight difference in notation and using the exponential map and parallel transport instead of a retraction and a corresponding vector transport.
Because of radially Lipschitz continuous differentiability of $f \circ \mathcal{R}$, we have constant $L_{RL}$ instead of $L_f$.
The inequalities \cref{Dec-inner} and \cref{Dec-outer} were proven in \cite[Lemma 4.1]{gutman2022coordinate} with the mentioned changes.
\end{proof}

In the following theorem, we give a convergence rate for the global convergence, then we prove a tighter global rate of convergence for retraction-convex functions.

\begin{theorem}[Global convergence] \label{thm.local-conv}
	Assume $f$ has restricted Lipschitz-type gradient and is lower bounded. Then for the sequence generated by  \cref{alg:TSD}, we have $\lVert \nabla f(x^t) \rVert^2_{x^t}\rightarrow 0$, and we have the following as the rate of convergence:
	\begin{equation} \label{local-conv}
		\min_{i = \{1,...,t\}} \lVert \nabla f(x^{i-1}) \rVert_{x^{i-1}} \leq \sqrt{\frac{1}{t} \Big ( f(x^0) - f(x^t) \Big ) 4L_{max}(1+Cm)} \ .
	\end{equation}
\end{theorem}
\begin{proof}
	From \cref{Dec-outer}, we have
	\begin{equation} \label{outer-decrease} 
		f(x^0) - f(x^t) \geq \frac{1}{4L_{max}(1+Cm)} \sum_{i=1}^{t} \lVert \nabla f(x^{i-1}) \rVert^2_{x^{i-1}} \ ,
	\end{equation}
	So, from the fact that $f$ is lower bounded, we can easily conclude that
	\begin{displaymath}
		t\rightarrow \infty \quad : \quad \lVert \nabla f(x^t) \rVert^2_{x^t}\rightarrow 0. 
	\end{displaymath}
        The inequality \cref{local-conv} can be derived effortlessly from \cref{outer-decrease}.
\end{proof}

\begin{theorem}[Tighter global convergence]
\label{thm.convergence}
	Let $f: \mathcal{M} \rightarrow \mathbb{R}$ be a retraction-convex function and the Sufficient Decrease lemma holds for the sequence $\{x^t\} \subset \mathcal{M}$. If we denote the sufficient decrease constant $1/K$, i.e.
	\begin{displaymath}
		\frac{L_{\min }^{2}}{4 L_{\max }\left(L_{\min }^{2}+L_{RL}^{2} m(m-1)\right)} = \frac{1}{K},
	\end{displaymath}
	then for $t > 1$ and $\eta_t = \mathcal{R}^{-1}_{x^t}(x^*)$
	\begin{equation} \label{convergence}
		f(x^{t+1}) - f^* \le \frac{K \|\eta_t\|^{2}_{x^t}(f(x^1)-f^*)}{K \|\eta_t\|^{2}_{x^t}+t(f(x^1)-f^*)}
	\end{equation}
\end{theorem}

\begin{proof}
	From the retraction-convexity of function $f$ we have
	\begin{displaymath}
		0 \leq f(x^t) - f(x^*) \leq - \langle \nabla f(x^t), \eta_t \rangle_{x^t} \leq \| \nabla f(x^t) \|_{x^t} \|\eta_t\|_{x^t} .
	\end{displaymath}
	After combining that with the result of Sufficient Decrease lemma \ref{sufficient lemma}, we will have
	\begin{displaymath}
		f(x^t) - f(x^{t+1}) \geq \frac{1}{K \|\eta_t\|_{x^t}^2}\big[ f(x^t) - f(x^*)\big]^2 .
	\end{displaymath}
	We know that for every real-valued decreasing sequence $A_t$ if $A_t - A_{t+1} \ge \alpha A^{2}_{t}$ for some $\alpha$, then $ A_{t+1}\le\frac{A_1}{1+A_1 \alpha t}$.
	Using this on the above inequality, we reach the convergence bound \cref{convergence}.
\end{proof}

Transitivity for vector transports, i.e. $\mathcal{T}_x^y \mathcal{T}_y^z = \mathcal{T}_x^z$, does not hold for Riemannian manifolds in general. But due to the fact that each point and each tangent vector in a product manifold is represented by Cartesian products, we can obtain the constant $C$ in a simpler way than what has come in the proof of \cref{sufficient lemma}.  
For product manifold in the Tucker decomposition problem \cref{TD}, each orthogonal projection of the gradient is simply the gradient of the cost function w.r.t the variables of one of the manifolds in the product manifold, and therefore the gradient projection belongs to the tangent space of that manifold.
For a tangent vector satisfying $\xi_{y^0} = \mathcal{R}^{-1}_{y^0}(y^{i-1})$ which is the case for product manifolds, we have
\begin{align*}
	\| \nabla f(y^0) - S^{i-1} \nabla f(y^{i-1}) \|_{y^0}^2 & \le L_{RL}^2 \|\xi_{y^0} \|_{y^0}^2  \\
	& \le L_{RL}^2 \sum_{j=1}^{i-1} \Big\| -\frac{1}{L_j}P^{y^{j-1}}_j\nabla f(y^{j-1}) \Big\|_{y^{i-1}}^2 \\
	&\le \frac{L_{RL}^2}{L_{min}^2} \sum_{j=1}^{i-1}\Big\| P^{y^{j-1}}_j\nabla f(y^{j-1}) \Big\|_{y^{i-1}}^2,
\end{align*}
where in the second inequality we exploited triangular inequality on $T_{y_0}\mathcal{M}$ and the fact that vector transport is isometric.
So, the factor $m-1$ is removed from the rates of convergence in \cref{thm.local-conv} and \cref{thm.convergence}, thus they match the rates of convergences of the coordinate descent method in the Euclidean setting \cite{beck2013convergence}.

The Tucker decomposition problem \cref{TD} is not retraction-convex, so we can not use the result of \cref{thm.convergence} for it. But by  \cref{lip-grad} and the fact that the objective function is lower bounded, we reach the following corollary from \cref{thm.local-conv}.

\begin{corollary} \label{corollary}
The RPDC algorithm given in \cref{alg:RPCD} with $m=d$ has the same rate of global convergence as given in \cref{thm.local-conv}, wherein $C = L_{RL}^2/L^2_{min}$.
\end{corollary}

It is worth noting that the proof of convergence for the HOOI method which solves the same objective function was investigated in \cite{xu2018convergence}, but it did not provide a convergence rate. 

\section{Experimental Results}
In this section, we evaluate the performance of our proposed methods on  \emph{synthetic} and \emph{real} data. 
We implemented the proposed methods in Matlab\footnote{Implementation of RPCD and RPCD+ can be found via \url{https://github.com/utvisionlab/rpcd}} R2023a and also used available Matlab implementations of other methods for comparison. 
The experiments were performed on a laptop computer with an Intel Core-i7 12700H CPU and 32 GB of memory.
We used Matlab's default setting regarding the use of CPU cores usage.
The stepsize for RPCD and RPCD+ is set to one. 
For the tables,  $\epsilon$ is put to $0.001$ and for the figures, it is set to $10^{-5}$.
For the RPCD+ algorithm, we choose $\epsilon^\prime=\epsilon/10$. 

To establish a fair and objective evaluation of our approach relative to others, we employed a uniform stopping criterion across all algorithms, based on a measure of relative error delta. 
This metric quantifies the degree to which an algorithm can reduce relative error, ensuring that any observed differences in performance are not due to variations in stopping criteria.

The reported time for each method is the actual time that the method spends on the computations which leads to the update of the parameters, and the time for calculating the relative error or other computations are not taken into account. 
For the RPCD+ algorithm, we also take into the count the time needed to evaluate the relative error in the inner loop. 
For HOOI and ST-HOSVD implementations, we use \texttt{tucker\_als} and \texttt{hosvd} from the Tensor Toolbox \cite{koldatoolbox}\footnote{The latest release can be found here \url{https://gitlab.com/tensors/tensor_toolbox/-/releases}}, respectively.

\subsection{Synthetic Data}
In this part, we give the results for two cases of Tucker decomposition on dense random tensors. 
 In both cases, the elements of random matrices or tensors are drawn from a normal distribution with zero mean and unit variance. 
 In the first case, we generate each rank-$(r_1, r_2, r_3)$ tensor $\mathcal{A}_1$ from the $i$-mode production of a random core tensor in $\mathbb{R}^{r_1\times r_2 \times r_3}$ and 3 orthonormal matrices constructed by the \emph{QR-decomposition} of random $n_i$ by $r_i$ matrices. In the second case, which has more resemblance with real data with intrinsic low-rank representations, we construct each tensor $\mathcal{A}_2$ by adding noise to a low-rank tensor,
\begin{displaymath}
	\mathcal{A}_2 = \mathcal{L} / \| \mathcal{L} \|_F + 0.1\cdot\mathcal{N} / \| \mathcal{N} \|_F ,
\end{displaymath}
where $\mathcal{L}$ is a low-rank tensor generated similar to $\mathcal{A}_1$ and $\mathcal{N}$ is a tensor with random elements.

In both set of experiments, we set $r_1 = r_2 = r_3 = 5$. 
Because of the memory limitation, except in the last experiment we increase the dimension of just the first mode of the input dense tensor to have the performance comparison in higher dimensions.
In the last experiment, we uniformly increase all dimensions of the input tensor to demonstrate the effect of tensor volume on different methods.
Each experiment is repeated 5 times and the reported time is the average value.
The results can be seen in \cref{table: synthesis}.

\tabcolsep=0.06cm
\begin{table}[tb]
{\footnotesize
    \caption{Execution time comparison in seconds for the RPCD+, HOOI and ST-HOSVD methods in the low-rank ($\mathcal{A}_1$) and low-rank with noise ($\mathcal{A}_2$) settings.}
    \label{table: synthesis}
\begin{center}
    \begin{tabular}{| c | c c c c | c c c c| }
        \hline
        & \multicolumn{4}{c}{$\mathcal{A}_1$} \vline & \multicolumn{4}{c}{$\mathcal{A}_2$} \vline \\
        n             & RPCD+ & HOOI  & {\scriptsize ST-HOSVD} & {\scriptsize ST-HOSVD} & RPCD+ & HOOI  & {\scriptsize ST-HOSVD} & {\scriptsize ST-HOSVD} \\
                      &        &        & (eigs) &  (svds) &        &       & (eigs)   & (svds)
        \\ \hline
        [100,100,100] & 0.04  & 0.1  & \textbf{0.02}  & 0.06 & 0.05  & 0.1  & \textbf{0.02}  & 0.07     \\ \hline
        [1000,100,100] & 0.08  & 0.14  & \textbf{0.05}  & 0.18 & 0.08  & 0.14  & \textbf{0.05}  & 0.39    \\ \hline
        [10,000,100,100] & \textbf{0.38}  & 1.71  & 0.84  & 1.52 & \textbf{0.4}  & 1.72  & 0.86  & 4.22  \\ \hline
        [20,000,100,100] & \textbf{0.72}  & 5.52 & 2.70 & 3.04 & \textbf{0.74}  & 5.64 & 2.83 & 9.19  \\ \hline
        [40,000,100,100] & \textbf{1.33}  & 20.17 & 9.54 & 6.09 & \textbf{1.45}  & 19.75 & 9.72 & 16.90  \\ \hline
        [1000,1000,1000] & \textbf{4.52}  & 4.56 & 7.69 & 10.72 & \textbf{5.02}  & 5.27 & 9.34 & 92.3  \\ \hline
    \end{tabular}
\end{center}
}
\end{table}

As it can be seen in \cref{table: synthesis}, increasing the dimensionality leads to the emergence of a performance gap between the RPCD+ and HOOI algorithms.
This is because in RPCD+, the subproblem \cref{TD reformulate} is solved \emph{inexactly} by the \emph{QR-decomposition}; 
but in HOOI, it is solved almost \emph{exactly} by finding the \emph{eigenvectors} of the large matrix $Y_{(i)} Y_{(i)}^T$. In \texttt{tucker\_als}, it is done by the Matlab function \texttt{eigs} that uses the \emph{Lanczos} algorithm with the number of computations (the number of additions and multiplications) approximately equal to $T(4n_i^2r_i+8r_i^3)$, where $T$ is the number of iterations in the loop of the \emph{Lanczos} algorithm. 
Without loss of the generality, we analyze the computational complexity for the last block $i=d$. 
The number of computations of $\mathcal{Y}_d$ is equal to $2\sum_{k=1}^{d-1} \prod_{l=1}^k r_l \prod_{m=k}^d n_m$. 
Therefore in total, the number of computations of HOOI is approximately equal to $2\sum_{k=1}^{d-1} \prod_{l=1}^k r_l \prod_{m=k}^d n_m+n_d^2 \prod_{k=1}^{d-1} r_k+T(4n_d^2r+8r_d^3)$ which is higher than  RPCD/RPCD+  with the number of computations approximately equal to $2\sum_{k=1}^{d-1} \prod_{l=1}^k r_l \prod_{m=k}^d n_m+T'(4n_d \prod_{k=1}^{d} r_k+2r_d^2n_d-2r_d^3/3)$, where $T'$ is the number of iterations in the inner loop of RPCD+. 
RPCD+ solves the the subproblem \cref{TD reformulate} inexactly, therefore $T'$ is significantly smaller than $T$.

Both the RPCD+ and HOOI methods reach desirable relative error, zero in the first case and $10\%$ for the second case, in the same number of iterations. 
But as the cost of each iteration is less for the RPCD+ method, we observe a reduced computational time in total.

As the Tensor Toolbox implementation of the ST-HOSVD method is inefficient due to the using the \texttt{eig} function, we modified the \texttt{hosvd} function to utilize the more efficient \texttt{eigs} and \texttt{svds} functions.
As can be seen from \cref{table: synthesis}, the ST-HOSVD method performs exceptionally well in lower dimensions, but as the dimension increases, although it outperforms HOOI, it still lags behind RPCD+.

We also investigated the case of overestimating the multi-linear rank for the synthetic data. 
We observed that in such cases, unlike the HOOI method which gets very slow, the proposed method does not need much additional time to compute the factor matrices. 
For example, in decomposing a noisy tensor $\mathcal{X}\in \mathbb{R}^{10000 \times 100 \times 100}$ and multi-linear rank$-(5,5,5)$ with over estimated lower multi-linear rank$-(7,7,7)$, the required time for RPCD+ and HOOI methods went from .39 and 1.71 seconds to .41 and 3.16 seconds, respectively.

\subsection{Real Data}
In the first experiment of this subsection, we compare the RPCD+ and HOOI methods for compressing the images in \emph{Yale face database}\footnote{A 64x64 version can be found here \url{http://www.cad.zju.edu.cn/home/dengcai/Data/FaceData.html}}\cite{belhumeur1997eigenfaces}. 
This dataset contains 165 grayscale images of 15 individuals. 
There are 11 images per subject in different facial expressions or configuration, thus we have a dense tensor $\mathcal{X} \in \mathbb{R}^{64\times 64 \times 11 \times 15}$. For two levels of compression, we decompose $\mathcal{X}$ to three tucker tensor with multi-linear rank $(16,16,11,15)$ and $(8,8,11,15)$, respectively. 
The results are shown in \cref{fig:yale}.

\begin{figure}[tb] \label{fig:yale}
	\centering
	\subfloat[16x16]{\includegraphics[width=6.5cm]{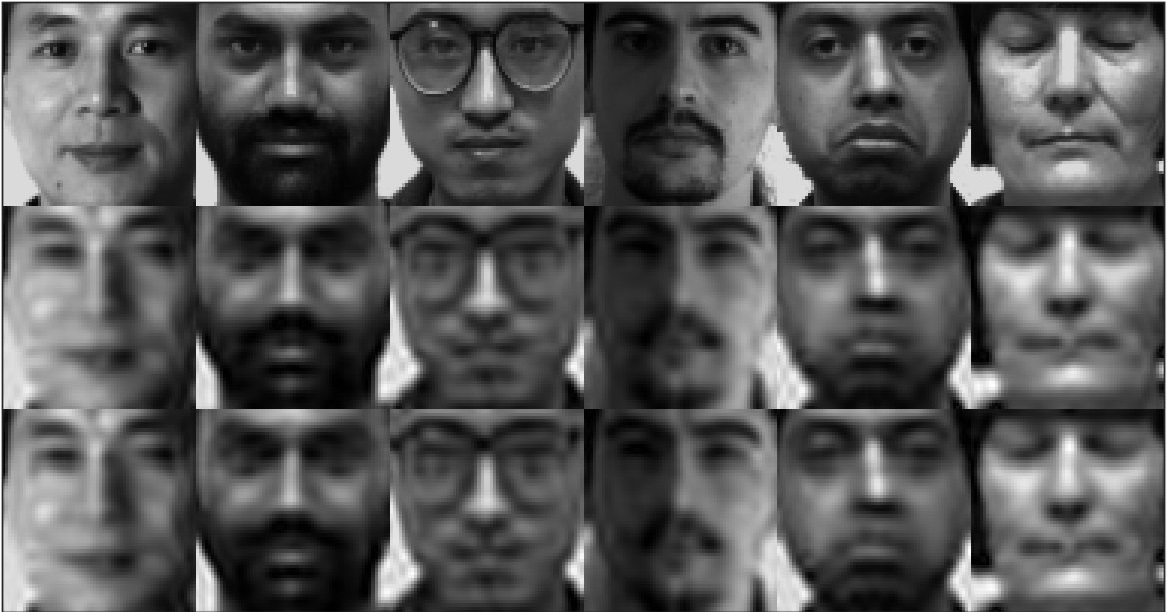}}
	\subfloat[8x8]{\includegraphics[width=6.5cm]{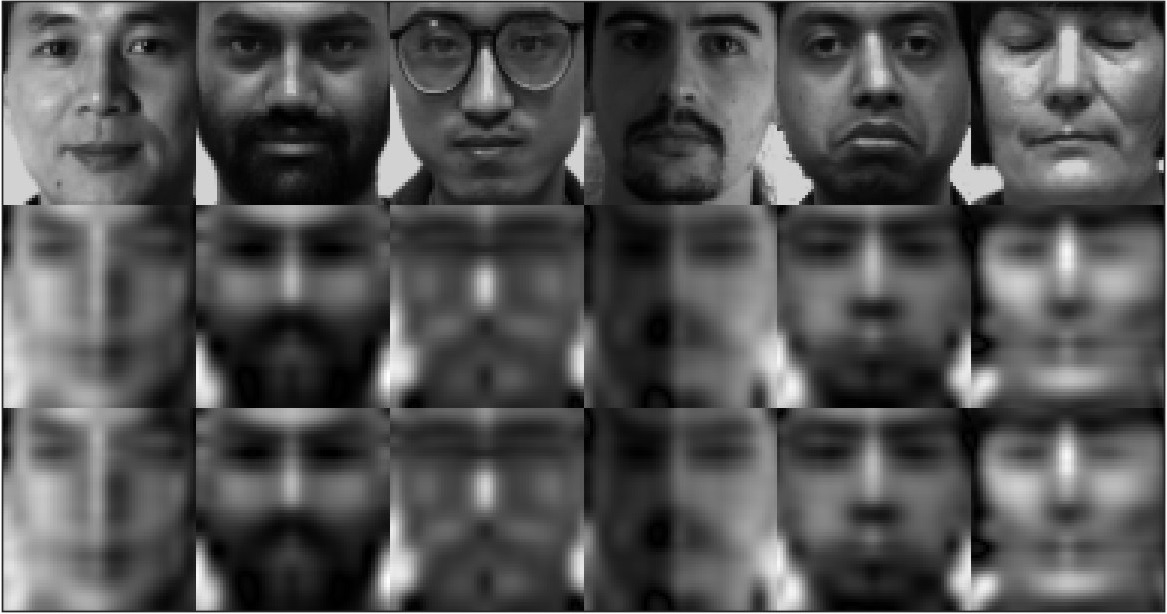}}
		\caption{Compression of Yale face database (1th row) with HOOI (2th row) and RPCD+ (3th row)}
\end{figure}

The first row in each figure contains the original images, the second and third rows contain the results of the compression using the HOOI and RPCD+ methods, respectively. The attained relative error for both algorithms are the same but RPCD+ is faster ($0.09$ vs $0.15$ seconds) for the case of $16\times16$. The difference in speed becomes larger ($0.06$ vs $0.12$ seconds), when we want to compress the data more, that is the case of $8\times8$.

In another comparison for the real data, we compare RPCD, RPCD+ and HOOI  with a newly introduced SVD-based method called D-Tucker\footnote{We use the Matlab implementation they provided in \url{https://datalab.snu.ac.kr/dtucker/}} \cite{jang2020d}. D-Tucker compresses the original tensor by performing randomized SVD on slices of the re-ordered tensor and then computes the orthogonal factor matrices and the core tensor using SVD. 
The paper \cite{jang2020d} reported that this method works well when the dimensions of a tensor is high in two modes, and the rest of the modes are low dimensional.
That is, for $\mathcal{X}_{re} \in \mathbb{R}^{I_1 \times I_2 \times K_3 \times \cdots \times K_d}$ where we have $I_1 \geq I_2 \gg K_3 \geq \cdots \geq K_d$, the algorithm needs to compute $L = K_3 \times \cdots \times K_d$ randomized SVDs.

The results from the real data are presented in \cref{table: real} and \cref{fig:real}.
As shown in \cref{table: real}, three different target ranks for each dataset are used. The first set of target ranks was chosen based on the D-Tucker paper \cite{jang2020d}. 
For the second set, we multiplied the first set of target ranks by the factor of five. 
For the third target ranks we aim for compression, around 10\% relative error, raising the ranks of the first set uniformly to achieve the desired quality. 
The first set of target ranks was used for the plots of \cref{fig:real}.
We initialized the factor matrices using the standard practice of setting the main diagonal to one and the rest to zero, as commonly done in iterative eigen-solvers. 
The reason for the discrepancy of timing results between the figure and table is the smaller stopping criterion chosen to produce the plots.  
For \cref{fig:real}, we allowed the algorithms to iterate more to clearly see the convergence behavior of different methods. 
The vertical axis in \cref{fig:real} is representing the difference between relative error of each algorithm and the relative error of the most accurate algorithm (a.k.a $E_{min}$). 
The plots are cut at $10^{-3}$ as smaller differences in relative error is minuscule and we can ignore them.

As shown in \cref{table: real}, the RPCD+ method consistently achieves a better final relative error than RPCD, thanks to its precision update process. 
Additionally, in some cases, RPCD+ is faster due to its smaller number of required iterations to converge. 
We observe that in the Air Quality and HSI datasets, D-Tucker is computationally advantageous, but as shown in \cref{fig:real}, this advantage is due to early stopping, resulting in poorer precision.
In contrast, RPCD+ and HOOI perform well in terms of speed and precision in the Yale and Coil-100 datasets, which have large $L$ values and the provided code for D-Tucker returns \texttt{``rank deficient warning''}.
For the third set of target ranks,
the implementation of D-Tucker did not work.
For Brainq and Air Quality datasets, we can see the advantage of RPCD as the cost of updating each factor matrix becomes higher for large target ranks.
In such cases, it is better to alternate between factor matrices instead of finding a better update for each of them.

Both RPCD+ and HOOI offer the best low multi-linear rank approximation, but HOOI is considerably slower in high-dimensional cases. It is worth noting that the slight difference in the relative Error between RPCD+ and HOOI reported in \cref{table: real} which is a bit in favor of HOOI is the equal stopping criterion used for the methods. Each step of HOOI decreased the error more than one step of RPCD+ as it can be seen in \cref{fig:real}. Therefore, it is natural to choose a lower stopping criterion for RPCD+ but we used equal stopping criterion to avoid the concern of its arbitrary selection.
An important observation from these experiments is that RPCD+ performs well in lower dimensions and offers superior performance in high-dimensional cases, as seen in the results from the synthetic data. 
Therefore, RPCD+ is a reliable general method for multi-linear data analysis.

\definecolor{Gray}{gray}{0.9}
\tabcolsep=0.03cm
\begin{table}[tb]
{\scriptsize
\caption{Execution time in seconds and relative error for the D-Tucker, RPCD, RPCD+ and HOOI methods on the real datasets.}
\label{table: real}
\begin{center}
    \renewcommand{\arraystretch}{1.2}
    \begin{tabular}{|>{\columncolor{Gray}}c|cc|cc|cc|cc|cc|}
        \hline
        Dataset & \multicolumn{2}{c}{Yale \cite{belhumeur1997eigenfaces}} \vline & \multicolumn{2}{c}{Brainq \cite{mitchell2008predicting}} \vline & \multicolumn{2}{c}{Air Quality\footnote{Downloaded from \url{https://datalab.snu.ac.kr/dtucker/}}} \vline& \multicolumn{2}{c}{HSI \cite{foster2006frequency}} \vline & \multicolumn{2}{c}{Coil-100 \cite{nene1996columbia}} \vline \\ \hline
        Dimension & \multicolumn{2}{c}{[64 64 11 15]} \vline & \multicolumn{2}{c}{[360 21764 9]} \vline & \multicolumn{2}{c}{[30648 376 6]} \vline& \multicolumn{2}{c}{[1021 1340 33 8]} \vline & \multicolumn{2}{c}{[128 128 72 100]} \vline \\ \hline
        Target Rank& \multicolumn{2}{c}{[5 5 5 5]} \vline & \multicolumn{2}{c}{[10 10 5]} \vline & \multicolumn{2}{c}{[10 10 5]} \vline& \multicolumn{2}{c}{[10 10 10 5]} \vline & \multicolumn{2}{c}{[5 5 5 5]} \vline \\ \hline
        \rowcolor{Gray}
        & time(s) & E(\%) & time(s) & E(\%) & time(s) & E(\%) & time(s) & E(\%) &  time(s) & E(\%)\\ \hline
        D-Tucker & 0.11 & 30.46 & \textbf{0.8} & \textbf{77.38} & \textbf{0.58} & 33.08 & \textbf{3.13} & 45.17 & 5.84 & 36.64 \\ \hline
        RPCD &  0.04 & 30.02 & 2.88 & 78.35 & 0.87 & 32.87 & 6.33 & 43.69 & 1.24 & 36.42   \\ \hline
        RPCD+ &  \textbf{0.02} & 29.93 & 2.81 & 77.87 & 0.84 & 32.74 & 4.20 & 43.48 & {\bf 0.72} & \textbf{36.35}   \\ \hline
        HOOI &  {\bf 0.02} & \textbf{29.92} & 41.09 & 77.92 & 33.82 & \textbf{32.72} & 4.31 & \textbf{43.42} & 0.74 & \textbf{36.35}  \\ \hline
        Target Rank& \multicolumn{2}{c}{[25 25 5 5]} \vline & \multicolumn{2}{c}{[50 50 5]} \vline & \multicolumn{2}{c}{[50 50 5]} \vline& \multicolumn{2}{c}{[50 50 10 5]} \vline & \multicolumn{2}{c}{[25 25 25 25]} \vline \\ \hline
        \rowcolor{Gray}
        & time(s) & E(\%) & time(s) & E(\%) & time(s) & E(\%) & time(s) & E(\%) &  time(s) & E(\%)\\ \hline
        D-Tucker & -- & -- & 3.93 & \textbf{60.78} & -- & -- & 12.18 & 32.60 & -- & -- \\ \hline
        RPCD &  0.06 & 26.84 & {\bf 2.75} & 67.86 & {\bf 1.51} & 24.44 & 9.28 & 32.18 & 2.77 & 24.37   \\ \hline
        RPCD+ &  \textbf{0.04} & 26.56 & 4.53 & 65.18 & {\bf 1.52} & 24.39 & {\bf 6.57} & 32.10 & 1.73 & 24.30   \\ \hline
        HOOI &  0.05 & \textbf{26.54} & 96.51 & 65.00 & 84.50 & {\bf 24.32} & {\bf 6.57} & \textbf{32.06} & {\bf 1.65} & \textbf{24.24}  \\ \hline
        Target Rank& \multicolumn{2}{c}{[17 17 11 15]} \vline & \multicolumn{2}{c}{[360 2700 9]} \vline & \multicolumn{2}{c}{[300 300 6]} \vline& \multicolumn{2}{c}{[220 220 33 8]} \vline & \multicolumn{2}{c}{[70 70 70 70]} \vline \\ \hline
        \rowcolor{Gray}
        & time(s) & E(\%) & time(s) & E(\%) & time(s) & E(\%) & time(s) & E(\%) &  time(s) & E(\%)\\ \hline
        D-Tucker & -- & -- & -- & -- & -- & -- & -- & -- & -- & -- \\ \hline
        RPCD &  0.08 & 10.42 & {\bf 65.75} & 10.30 & {\bf 8.05} & 10.89 & 27.38 & 10.52 & 14.58 & 10.15   \\ \hline
        RPCD+ & 0.06 & 10.37 & 101.56 & 10.16 & 14.34 & 10.84 & 22.87 & 10.47 & 10.76 & 10.02   \\ \hline
        HOOI &  \textbf{0.05} & \textbf{10.34} & 1243.86 & {\bf 10.12} & 395.23 & \textbf{10.83} & {\bf 17.96} & \textbf{10.44} & \textbf{4.58} & \textbf{9.93}  \\ \hline
    \end{tabular}
\end{center}
}
\end{table}
\begin{figure}[tbhp] \label{fig:real}
\centering
\subfloat[Yale Face]{%
	\resizebox{.32\textwidth}{!}{\input{tikzfig/yale.tex}}
}
\subfloat[Brainq]{%
  \resizebox{0.325\textwidth}{!}{\input{tikzfig/brainq.tex}}
  }%
\subfloat[Air Quality]{%
  \resizebox{0.325\textwidth}{!}{\input{tikzfig/airQuality.tex}}
  } \\
\subfloat[HSI]{%
  \resizebox{.312\textwidth}{!}{\input{tikzfig/hsi.tex}}
  }
\subfloat[Coil-100]{%
  \resizebox{.325\textwidth}{!}{\input{tikzfig/coil.tex}}
  }
 \caption{Convergence behavior of different methods for the real datasets. Y-axis is the difference between the relative error at each iteration and the best achieved relative error.}
\end{figure}
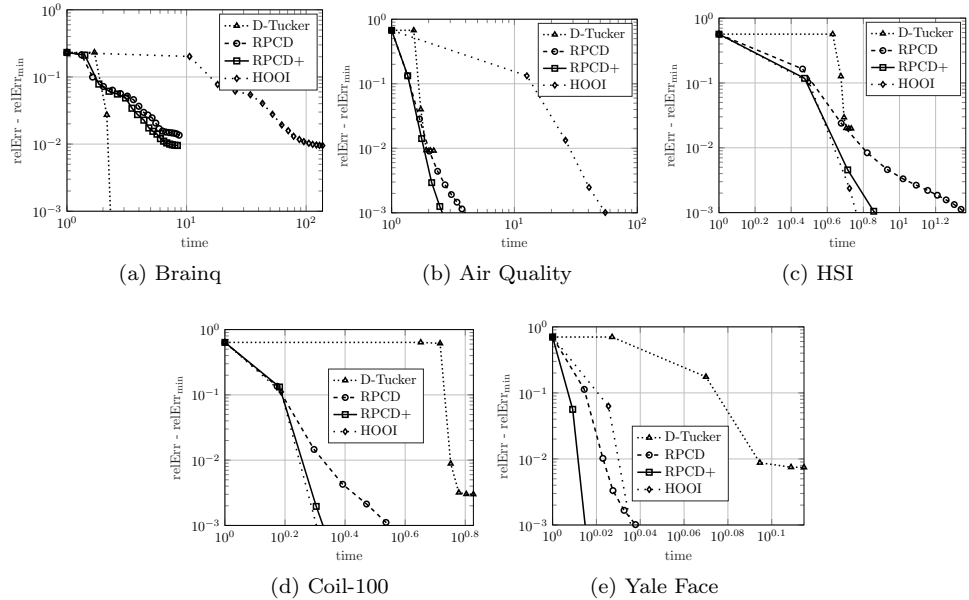

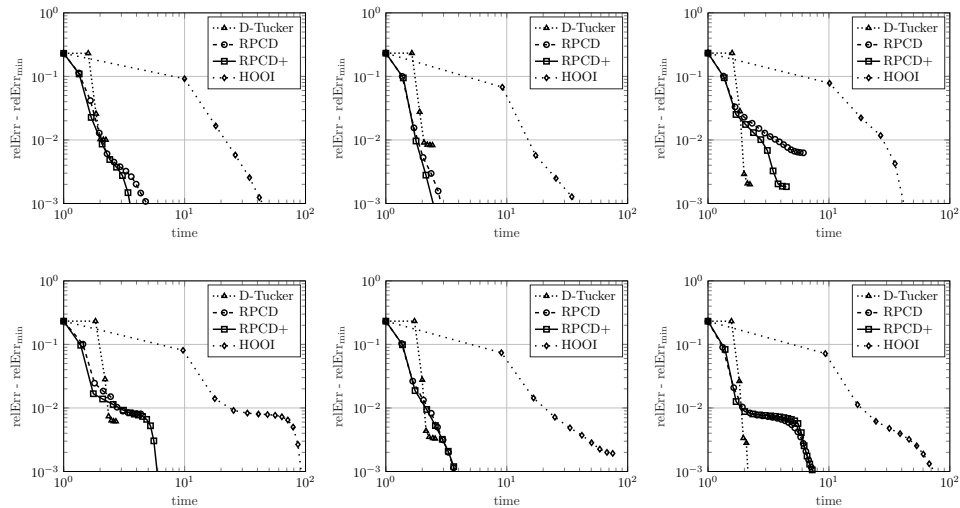
\begin{figure}[!tbhp] \label{fig:init}
\centering
\subfloat{%
  \resizebox{.32\textwidth}{!}{\input{tikzfig/brainq_1.tex}}
  }%
\subfloat{%
  \resizebox{.32\textwidth}{!}{\input{tikzfig/brainq_2.tex}}
  }
\subfloat{%
  \resizebox{.32\textwidth}{!}{\input{tikzfig/brainq_3.tex}}
  }\\
\subfloat{%
  \resizebox{.32\textwidth}{!}{\input{tikzfig/brainq_4.tex}}
  }
  \subfloat{%
  \resizebox{.32\textwidth}{!}{\input{tikzfig/brainq_5.tex}}
  }
  \subfloat{%
  \resizebox{.32\textwidth}{!}{\input{tikzfig/brainq_6.tex}}
  }
  \caption{Convergence behavior for decomposing the Brainq dataset using different random initializations.}
\end{figure}
 
For all datasets except Brainq, we observe almost identical convergence behavior when we start at different starting points. The effect of different initialization on the performance of different methods for Brainq can be seen in \cref{fig:init}.

\section{Conclusion}
In this paper, we introduced RPCD and its improved version RPCD+, first-order method solving the Tucker decomposition problem for  high-order, -dimensional dense tensors with the Riemannian coordinate descent method. 
For these methods, we constructed a Riemannian metric by incorporating the second order information of the reformulated cost function and the constraint. 
We proved a convergence rate for general tangent subspace descent on Riemannian manifolds, which for the special case of product manifolds like the Tucker decomposition matches the rate in the Euclidean setting. 
Experimental results showed that RPCD+ as a general method has competitive performance among competing methods for high-order, high-dimensional tensors. 

For a future work, it would be interesting to examine the RPCD method in solving tensor completion problems. Another interesting line of work would be to incorporate latent tensors between the original tensor $\mathcal{X}$ and the projected tensor $\mathcal{Y}$ for further reducing computation costs.

\section*{Acknowledgments}
The authors would like to express their  gratitude for the exceptional careful comments of the anonymous reviewers in all rounds of the review process that helped improving the quality of the paper.

\bibliographystyle{siamplain}
\bibliography{references}

\end{document}

%% file: ex_shared.tex
\usepackage{amsfonts}
\usepackage{graphicx}
\usepackage[noend]{algpseudocode}
\usepackage[caption=false]{subfig}
\usepackage{multirow}
\usepackage{color, colortbl}
\usepackage{pgfplots}
\usepackage{amsopn}

\ifpdf
  \DeclareGraphicsExtensions{.eps,.pdf,.png,.jpg}
\else
  \DeclareGraphicsExtensions{.eps}
\fi

\headers{Riemannian Preconditioned Coordinate Descent}{M. Hamed and R. Hosseini}

\title{Riemannian Preconditioned Coordinate Descent\\ for low multi-linear rank approximation\thanks{Submitted to the editors \today.
\funding{The first author was partially supported by the Research Foundation Flanders (FWO) research project G081222N and UA BOF DocPRO4 projects with ID 46929 and 48996.}}}

\author{Mohammad Hamed\footnotemark[3] \thanks{Department of Mathematics, University of Antwerp, Antwerp, Belgium (\email{mohammad.hamed@uantwerp.be})}
\and Reshad Hosseini \thanks{School of ECE, College of Engineering, University of Tehran, Tehran, Iran (\email{mohamad.hamed@ut.ac.ir}, \email{reshad.hosseini@ut.ac.ir}).}}


%% file: tikzfig/euclidean_norm.tex
%
%
\begin{tikzpicture}

\begin{axis}[%
scale=0.75,
width=4.521in,
height=3.566in,
at={(0.758in,0.481in)},
scale only axis,
xmin=0,
xmax=55,
xlabel style={font=\color{white!15!black}},
xlabel={Number of Iterations},
ymode=log,
ymin=1e-05,
ymax=1,
yminorticks=true,
ylabel style={font=\color{white!15!black}},
ylabel={Relative Error},
axis background/.style={fill=white},
xmajorgrids,
ymajorgrids,
]
\addplot [color=black, dashed, line width=1.5pt, forget plot]
  table[row sep=crcr]{%
1	0.999961375054211\\
2	0.999751652353573\\
3	0.881291174839451\\
4	0.817720159723859\\
5	0.639616618335004\\
6	0.433542805041518\\
7	0.362461531396614\\
8	0.336454194291006\\
9	0.320303984211221\\
10	0.302699774414071\\
11	0.276716041925474\\
12	0.236796565476427\\
13	0.182945141472561\\
14	0.125964022233569\\
15	0.0807999096323698\\
16	0.0528636605139473\\
17	0.038056863391554\\
18	0.0305569751260573\\
19	0.0264424820077495\\
20	0.0237828955114538\\
21	0.0217757235704736\\
22	0.0201040901177035\\
23	0.0186379187404126\\
24	0.0173181159876817\\
25	0.0161139254145579\\
26	0.015006938050568\\
27	0.0139883554389436\\
28	0.0130430019061374\\
29	0.00989552538809205\\
30	0.00923499604952596\\
31	0.00861755371617661\\
32	0.00835743134302494\\
33	0.00835348975717559\\
34	0.00835342819897797\\
35	0.00835330509530684\\
36	0.00835329740070251\\
37	0.00835329734058847\\
38	0.00835329722037194\\
39	0.00835329720534487\\
40	0.00835329720346649\\
41	0.00835329720346557\\
42	0.00835329720346557\\
43	0.00835329720346557\\
44	0.00835329720346557\\
45	0.00835329720346557\\
46	0.00835329720346557\\
47	0.00835329720346557\\
48	0.00835329720346557\\
49	0.00835329720346557\\
50	0.00835329720346557\\
51	0.00835329720346557\\
};
\addplot [color=black, dashed, line width=1.5pt, forget plot]
  table[row sep=crcr]{%
1	0.999961592874055\\
2	0.999860733202966\\
3	0.906910450373507\\
4	0.746325613213088\\
5	0.542087981827366\\
6	0.411975273101834\\
7	0.380134692077829\\
8	0.356443516082557\\
9	0.332465694602043\\
10	0.304849366620364\\
11	0.269571390164833\\
12	0.223719266042969\\
13	0.170123784607446\\
14	0.118195255511391\\
15	0.0768715958749651\\
16	0.0484206748467853\\
17	0.0302828736647975\\
18	0.0190574788992902\\
19	0.0121470753755934\\
20	0.00786628487367402\\
21	0.00518175768876997\\
22	0.00364744190899255\\
23	0.00340773979309591\\
24	0.00338223742048735\\
25	0.00338218840078352\\
26	0.00338218798914685\\
27	0.00338218798434117\\
28	0.00338218797192658\\
29	0.00338218797182608\\
30	0.00338218797180239\\
31	0.00338218797179931\\
32	0.0033821879717993\\
33	0.00338218797179929\\
34	0.00338218797179929\\
35	0.00338218797179929\\
36	0.00338218797179928\\
37	0.00338218797179928\\
38	0.00338218797179927\\
39	0.00338218797179927\\
40	0.00338218797179927\\
41	0.00338218797179926\\
42	0.00338218797179926\\
43	0.00338218797179926\\
44	0.00338218797179926\\
45	0.00338218797179926\\
46	0.00338218797179926\\
47	0.00338218797179927\\
48	0.00338218797179927\\
49	0.00338218797179927\\
50	0.00338218797179927\\
51	0.00338218797179927\\
};
\addplot [color=black, dashed, line width=1.5pt, forget plot]
  table[row sep=crcr]{%
1	0.999925847631146\\
2	0.998882359067722\\
3	0.801744760338721\\
4	0.853406407240751\\
5	0.491898378399131\\
6	0.357105493684929\\
7	0.308268693435908\\
8	0.256884690489255\\
9	0.19649576651288\\
10	0.156995583158483\\
11	0.128050354219898\\
12	0.105209531986219\\
13	0.0866647687848947\\
14	0.0713633673253656\\
15	0.0586904915672854\\
16	0.0482135667887285\\
17	0.0395787294595906\\
18	0.0324808010499267\\
19	0.0266564337084632\\
20	0.0242347087164708\\
21	0.0198949953714733\\
22	0.0197845395185765\\
23	0.019770530903306\\
24	0.0197688217469631\\
25	0.0197683726342276\\
26	0.0197682658236744\\
27	0.0197682388439991\\
28	0.0197682380067115\\
29	0.0197682379937143\\
30	0.0197682379928577\\
31	0.0197682379912321\\
32	0.0197682379912289\\
33	0.0197682379912289\\
34	0.0197682379912289\\
35	0.0197682379912289\\
36	0.0197682379912289\\
37	0.0197682379912289\\
38	0.0197682379912289\\
39	0.0197682379912289\\
40	0.0197682379912289\\
41	0.0197682379912289\\
42	0.0197682379912289\\
43	0.0197682379912289\\
44	0.0197682379912289\\
45	0.0197682379912289\\
46	0.0197682379912289\\
47	0.0197682379912289\\
48	0.0197682379912289\\
49	0.0197682379912289\\
50	0.0197682379912289\\
51	0.0197682379912289\\
};
\addplot [color=black, dashed, line width=1.5pt, forget plot]
  table[row sep=crcr]{%
1	0.999933865381183\\
2	0.999664435145737\\
3	0.919295940860244\\
4	0.955977826460827\\
5	0.689795192164775\\
6	0.446832706703627\\
7	0.326371564705159\\
8	0.245757745380031\\
9	0.176668210131208\\
10	0.120471583282537\\
11	0.077956776304032\\
12	0.0487464716467801\\
13	0.0300304152827969\\
14	0.0184302602758104\\
15	0.0113244655677583\\
16	0.00698123655932656\\
17	0.00432200670877796\\
18	0.00268835034295261\\
19	0.00168694206975888\\
20	0.00137647612263546\\
21	0.00134201838310959\\
22	0.00108790424130284\\
23	0.000686464538436334\\
24	0.000433031960193088\\
25	0.000276026415672058\\
26	0.000176765214967471\\
27	0.000113610073901577\\
28	7.33545646697485e-05\\
29	4.75372350266517e-05\\
30	4.64958790348712e-05\\
31	3.83543156233528e-05\\
32	2.49639839183567e-05\\
33	2.49300647287539e-05\\
34	2.49279477687027e-05\\
35	2.49279394993049e-05\\
36	2.49279384656353e-05\\
37	2.49279383364312e-05\\
38	2.49279383359194e-05\\
39	2.49279383359168e-05\\
40	2.49279383359162e-05\\
41	2.49279383359222e-05\\
42	2.49279383359171e-05\\
43	2.49279383359172e-05\\
44	2.49279383359179e-05\\
45	2.49279383359146e-05\\
46	2.49279383359134e-05\\
47	2.49279383359055e-05\\
48	2.49279383359018e-05\\
49	2.49279383359011e-05\\
50	2.49279383359027e-05\\
51	2.49279383358992e-05\\
};
\addplot [color=black, dashed, line width=1.5pt, forget plot]
  table[row sep=crcr]{%
1	0.999915534371111\\
2	0.999152571708018\\
3	0.84681211316569\\
4	0.825768930933964\\
5	0.71781361743498\\
6	0.512736994646571\\
7	0.435493246531548\\
8	0.338827183628755\\
9	0.260480645457142\\
10	0.191574483575492\\
11	0.148752087957796\\
12	0.111947094797023\\
13	0.0822710201268752\\
14	0.0595378349276236\\
15	0.0427062833516783\\
16	0.0304901623612278\\
17	0.0217187389886292\\
18	0.0154550262926179\\
19	0.0110169404333493\\
20	0.0109880578742896\\
21	0.0109763009018634\\
22	0.0109761090495183\\
23	0.0109760859059615\\
24	0.010976085143218\\
25	0.0109760851136883\\
26	0.0109760851136856\\
27	0.0109760851136856\\
28	0.0109760851136856\\
29	0.0109760851136856\\
30	0.0109760851136856\\
31	0.0109760851136856\\
32	0.0109760851136856\\
33	0.0109760851136856\\
34	0.0109760851136856\\
35	0.0109760851136856\\
36	0.0109760851136856\\
37	0.0109760851136856\\
38	0.0109760851136856\\
39	0.0109760851136856\\
40	0.0109760851136856\\
41	0.0109760851136856\\
42	0.0109760851136856\\
43	0.0109760851136856\\
44	0.0109760851136856\\
45	0.0109760851136856\\
46	0.0109760851136856\\
47	0.0109760851136856\\
48	0.0109760851136856\\
49	0.0109760851136856\\
50	0.0109760851136856\\
51	0.0109760851136856\\
};
\addplot [color=black, dashed, line width=1.5pt, forget plot]
  table[row sep=crcr]{%
1	0.999978560182107\\
2	0.999941690557732\\
3	0.973190989110452\\
4	0.724601094002597\\
5	0.339457709115574\\
6	0.155697513101291\\
7	0.0917877864707282\\
8	0.0607261043114398\\
9	0.0416754882152235\\
10	0.0289109299508032\\
11	0.0201558019127245\\
12	0.0141024859436712\\
13	0.00989689657963418\\
14	0.00696361292907133\\
15	0.00491074727661911\\
16	0.00346976745657748\\
17	0.00245567703565909\\
18	0.00174042089735234\\
19	0.00123497461264637\\
20	0.000877211024821679\\
21	0.000623627864450061\\
22	0.000443675384869107\\
23	0.000315845801054456\\
24	0.000224964347436737\\
25	0.000160304677478401\\
26	0.000114272833677246\\
27	0.000105502447811056\\
28	0.000104984897953689\\
29	0.000104868247437039\\
30	0.000104635206785662\\
31	0.000102775264008052\\
32	0.000102637860635818\\
33	0.000102630687045743\\
34	0.000102623559236289\\
35	0.000102621432469685\\
36	0.000102617846992467\\
37	0.000102616916536436\\
38	0.000102616914454071\\
39	0.00010261691445345\\
40	0.000102616914453453\\
41	0.000102616914453448\\
42	0.000102616914453448\\
43	0.000102616914453445\\
44	0.000102616914453442\\
45	0.000102616914453437\\
46	0.000102616914453435\\
47	0.000102616914453432\\
48	0.000102616914453432\\
49	0.000102616914453427\\
50	0.000102616914453426\\
51	0.000102616914453423\\
};
\addplot [color=black, dashed, line width=1.5pt, forget plot]
  table[row sep=crcr]{%
1	0.999874825489279\\
2	0.996310795649683\\
3	0.8257226686572\\
4	0.854342852924113\\
5	0.448957136157046\\
6	0.27976983066398\\
7	0.189236368097036\\
8	0.133560608542915\\
9	0.0917214944712433\\
10	0.0613528457445066\\
11	0.0408291118600376\\
12	0.0275424672240564\\
13	0.0191065911578806\\
14	0.0137713064014464\\
15	0.0103633203840367\\
16	0.00812758046520087\\
17	0.00659558729718121\\
18	0.00548795162598541\\
19	0.00464365453107404\\
20	0.0039713448628639\\
21	0.00341870930677509\\
22	0.00295470940056603\\
23	0.00255985539840321\\
24	0.00222104278929041\\
25	0.00192883965823838\\
26	0.000196662710947151\\
27	0.000157268947591378\\
28	0.000112141981978332\\
29	0.000108677091406287\\
30	8.34956506906453e-05\\
31	8.14464529642361e-05\\
32	8.12006894245571e-05\\
33	8.11701237097119e-05\\
34	8.11090400089477e-05\\
35	8.06216113466339e-05\\
36	8.05612952907907e-05\\
37	8.05537652449752e-05\\
38	8.05387080877821e-05\\
39	8.05368265334535e-05\\
40	8.05217753077201e-05\\
41	8.05198944923479e-05\\
42	8.05048491776713e-05\\
43	8.03845628987807e-05\\
44	7.94271707126808e-05\\
45	7.20918387619534e-05\\
46	4.16566262600348e-05\\
47	2.83100112833685e-05\\
48	1.94362402472777e-05\\
49	1.87501574034171e-05\\
50	1.51784752630695e-05\\
51	1.00446863109339e-05\\
};
\addplot [color=black, dashed, line width=1.5pt, forget plot]
  table[row sep=crcr]{%
1	0.999952325220363\\
2	0.999647574057744\\
3	0.915291486048854\\
4	0.947172582080794\\
5	0.750550613324673\\
6	0.561877074380769\\
7	0.416232100421554\\
8	0.286077830946339\\
9	0.163159729292405\\
10	0.114024384739225\\
11	0.0793037800287238\\
12	0.0538408566184384\\
13	0.0360510486967791\\
14	0.0239808090894342\\
15	0.0159069365498118\\
16	0.0105398214337369\\
17	0.00698101558182361\\
18	0.00638799183647032\\
19	0.00631978254104928\\
20	0.00625307955998326\\
21	0.00612115325375487\\
22	0.00605648190918162\\
23	0.00580077332806811\\
24	0.00573918114586142\\
25	0.00573822648752018\\
26	0.00573810782642621\\
27	0.00573810781186351\\
28	0.00573810781141272\\
29	0.00573810781141249\\
30	0.00573810781141249\\
31	0.00573810781141249\\
32	0.00573810781141249\\
33	0.00573810781141249\\
34	0.00573810781141249\\
35	0.00573810781141249\\
36	0.00573810781141249\\
37	0.00573810781141249\\
38	0.00573810781141249\\
39	0.00573810781141249\\
40	0.00573810781141249\\
41	0.00573810781141249\\
42	0.00573810781141249\\
43	0.00573810781141249\\
44	0.00573810781141249\\
45	0.00573810781141249\\
46	0.00573810781141249\\
47	0.00573810781141249\\
48	0.00573810781141249\\
49	0.00573810781141249\\
50	0.00573810781141249\\
51	0.00573810781141249\\
};
\addplot [color=black, dashed, line width=1.5pt, forget plot]
  table[row sep=crcr]{%
1	0.999907997608546\\
2	0.998826572113447\\
3	0.892286602084825\\
4	0.821922496838677\\
5	0.783885303511394\\
6	0.738590452233502\\
7	0.729274993826511\\
8	0.720623350943057\\
9	0.700705080583904\\
10	0.681064943170263\\
11	0.646789261767942\\
12	0.582440680935525\\
13	0.496911559342667\\
14	0.432430166938466\\
15	0.392378861300617\\
16	0.367615749116424\\
17	0.349320878078097\\
18	0.331336334030245\\
19	0.310751454868697\\
20	0.290566198241419\\
21	0.270346442539592\\
22	0.238862384781653\\
23	0.194249508296127\\
24	0.163759332160443\\
25	0.133960350359764\\
26	0.10654771409416\\
27	0.082844221973065\\
28	0.0633795078731506\\
29	0.0479815672469766\\
30	0.0360946298851203\\
31	0.0270537892081931\\
32	0.0265167807094151\\
33	0.0264126926318339\\
34	0.026412641510203\\
35	0.0264125911291311\\
36	0.0264125848007224\\
37	0.0264125848004709\\
38	0.0264125847997021\\
39	0.0264125847996883\\
40	0.0264125847996402\\
41	0.0264125847996402\\
42	0.0264125847996401\\
43	0.0264125847996401\\
44	0.0264125847996401\\
45	0.0264125847996401\\
46	0.0264125847996401\\
47	0.0264125847996401\\
48	0.0264125847996401\\
49	0.0264125847996401\\
50	0.0264125847996401\\
51	0.0264125847996401\\
};
\addplot [color=black, dashed, line width=1.5pt, forget plot]
  table[row sep=crcr]{%
1	0.999857446629984\\
2	0.998140912211981\\
3	0.893242439885521\\
4	0.825906229693812\\
5	0.638744110868659\\
6	0.446109296681359\\
7	0.357685021615756\\
8	0.328164696227874\\
9	0.29995516925706\\
10	0.270851613706721\\
11	0.235277521096806\\
12	0.188086972739189\\
13	0.133519940026404\\
14	0.0818492129404811\\
15	0.0473091998342927\\
16	0.0267893162913945\\
17	0.0155101432904331\\
18	0.00953655600354003\\
19	0.00643543499314653\\
20	0.00480851077432211\\
21	0.00389691862997815\\
22	0.00372568549992245\\
23	0.00372424176496571\\
24	0.00372419668872438\\
25	0.00372419651265022\\
26	0.00372419644928189\\
27	0.0037241964485941\\
28	0.00372419644847032\\
29	0.00372419644846898\\
30	0.00372419644846897\\
31	0.00372419644846896\\
32	0.00372419644846895\\
33	0.00372419644846895\\
34	0.00372419644846894\\
35	0.00372419644846893\\
36	0.00372419644846892\\
37	0.00372419644846892\\
38	0.00372419644846892\\
39	0.00372419644846891\\
40	0.00372419644846891\\
41	0.00372419644846891\\
42	0.0037241964484689\\
43	0.0037241964484689\\
44	0.0037241964484689\\
45	0.00372419644846889\\
46	0.00372419644846889\\
47	0.00372419644846889\\
48	0.0037241964484689\\
49	0.0037241964484689\\
50	0.0037241964484689\\
51	0.0037241964484689\\
};
\end{axis}

\end{tikzpicture}%

%% file: tikzfig/yale.tex
%
%
\begin{tikzpicture}

\begin{axis}[%
scale = 0.55,
width=4.521in,
height=3.566in,
at={(0.758in,0.481in)},
scale only axis,
xmode=log,
xmin=1,
xmax=1.273862,
xminorticks=true,
xlabel style={font=\color{white!15!black}},
xlabel={time(s)},
ymode=log,
ymin=0.001,
ymax=1,
yminorticks=true,
ylabel style={font=\color{white!15!black}},
ylabel={$\text{relErr - relErr}_{\text{min}}$},
axis background/.style={fill=white},
xmajorgrids,
ymajorgrids,
legend style={legend cell align=left, align=left, at={(0.7,0.5)}, draw=white!15!black}
]
\addplot [color=black, dotted, line width=1.1pt, mark=triangle, mark options={solid, black}]
  table[row sep=crcr]{%
    1.0000 0.7008\\
    1.0978 0.7008\\
    1.1730 0.6441\\
    1.2215 0.0287\\
    1.2416 0.0077\\
    1.2542 0.0066\\
    1.2646 0.0066\\
};
\addlegendentry{D-Tucker}

\addplot [color=black, dashed, line width=1.1pt, mark=o, mark options={solid, black}]
  table[row sep=crcr]{%
    1.0000 0.7006\\
    1.0227 0.1130\\
    1.0357 0.0102\\
    1.0399 0.0033\\
    1.0452 0.0017\\
    1.0505 0.0010\\
};
\addlegendentry{RPCD}

\addplot [color=black, solid, line width=1.1pt, mark=square, mark options={solid, black}]
  table[row sep=crcr]{%
    1.0000 0.7006\\
    1.0097 0.0563\\
    1.0174 0.0008\\
    1.0230 0.0001\\
};
\addlegendentry{RPCD+}

\addplot [color=black, loosely dotted, line width=1.1pt, mark=diamond, mark options={solid, black}]
  table[row sep=crcr]{%
    1.0000 0.7006\\
    1.0428 0.0629\\
    1.0638 0.0007\\
    1.0834 0.0\\
};
\addlegendentry{HOOI}

\end{axis}

\end{tikzpicture}%

%% file: tikzfig/brainq.tex
%
%
\begin{tikzpicture}

\begin{axis}[%
scale = 0.55,
width=4.521in,
height=3.566in,
at={(0.758in,0.481in)},
scale only axis,
xmode=log,
xmin=1,
xmax=45.1218716,
xminorticks=true,
xlabel style={font=\color{white!15!black}},
xlabel={time(s)},
ymode=log,
ymin=0.001,
ymax=1,
yminorticks=true,
ylabel style={font=\color{white!15!black}},
ylabel={$\text{relErr - relErr}_{\text{min}}$},
axis background/.style={fill=white},
xmajorgrids,
ymajorgrids,
legend style={legend cell align=left, align=left, at={(0.85,0.4)}, draw=white!15!black}
]
\addplot [color=black, dotted, line width=1.1pt, mark=triangle, mark options={solid, black}]
  table[row sep=crcr]{%
    1.0000 0.2307\\
    1.4522 0.2307\\
    1.5819 0.0270\\
    1.6784 0.0013\\
    1.7407 0.0003\\
    1.7925 0.0\\
};
\addlegendentry{D-Tucker}

\addplot [color=black, dashed, line width=1.1pt, mark=o, mark options={solid, black}]
  table[row sep=crcr]{%
    1.0000 0.2307\\
    1.2038 0.2126\\
    1.4020 0.0987\\
    1.5983 0.0711\\
    1.7913 0.0626\\
    1.9851 0.0555\\
    2.1672 0.0508\\
    2.3548 0.0450\\
    2.5454 0.0357\\
    2.7362 0.0288\\
    2.9292 0.0262\\
    3.1342 0.0237\\
    3.3256 0.0196\\
    3.5148 0.0161\\
    3.7128 0.0147\\
    3.9031 0.0142\\
};
\addlegendentry{RPCD}

\addplot [color=black, solid, line width=1.1pt, mark=square, mark options={solid, black}]
  table[row sep=crcr]{%
    1.0000 0.2307\\
    1.2119 0.2102\\
    1.4256 0.0774\\
    1.6232 0.0603\\
    1.8281 0.0547\\
    2.0547 0.0477\\
    2.2556 0.0333\\
    2.4575 0.0266\\
    2.6852 0.0218\\
    2.8952 0.0165\\
    3.0936 0.0146\\
    3.3058 0.0132\\
    3.5110 0.0111\\
    3.7100 0.0101\\
    3.9150 0.0095\\
};
\addlegendentry{RPCD+}

\addplot [color=black, loosely dotted, line width=1.1pt, mark=diamond, mark options={solid, black}]
  table[row sep=crcr]{%
    1.0000 0.2307\\
    5.3127 0.2009\\
    9.2605 0.0764\\
    13.0619 0.0606\\
    17.1890 0.0534\\
    21.3078 0.0395\\
    25.4161 0.0269\\
    29.4462 0.0184\\
    33.4792 0.0149\\
    37.2622 0.0122\\
    41.0513 0.0108\\
    44.8140 0.0100\\
};
\addlegendentry{HOOI}

\end{axis}

\end{tikzpicture}%

%% file: tikzfig/airQuality.tex
%
%
\begin{tikzpicture}

\begin{axis}[%
scale = 0.55,
width=4.521in,
height=3.566in,
at={(0.758in,0.481in)},
scale only axis,
xmode=log,
xmin=1,
xmax=36,
xminorticks=true,
xlabel style={font=\color{white!15!black}},
xlabel={time(s)},
ymode=log,
ymin=0.001,
ymax=1,
yminorticks=true,
ylabel style={font=\color{white!15!black}},
ylabel={$\text{relErr - relErr}_{\text{min}}$},
axis background/.style={fill=white},
xmajorgrids,
ymajorgrids,
legend style={legend cell align=left, align=left, draw=white!15!black}
]
\addplot [color=black, dotted, line width=1.1pt, mark=triangle, mark options={solid, black}]
  table[row sep=crcr]{%
    1.0000 0.6727\\
    1.3606 0.6727\\
    1.4754 0.0420\\
    1.5735 0.0143\\
    1.6490 0.0142\\
};
\addlegendentry{D-Tucker}

\addplot [color=black, dashed, line width=1.1pt, mark=o, mark options={solid, black}]
  table[row sep=crcr]{%
    1.0000 0.6727\\
    1.1595 0.1325\\
    1.3143 0.0280\\
    1.4623 0.0086\\
    1.6133 0.0039\\
    1.7611 0.0022\\
    1.9092 0.0014\\
};
\addlegendentry{RPCD}

\addplot [color=black, solid, line width=1.1pt, mark=square, mark options={solid, black}]
  table[row sep=crcr]{%
    1.0000 0.6727\\
    1.1487 0.1325\\
    1.3143 0.0137\\
    1.4717 0.0025\\
    1.6253 0.0008\\
    1.7771 0.0002\\
};
\addlegendentry{RPCD+}

\addplot [color=black, loosely dotted, line width=1.1pt, mark=diamond, mark options={solid, black}]
  table[row sep=crcr]{%
    1.0000 0.6727\\
    6.8816 0.1325\\
    13.2426 0.0129\\
    20.4412 0.0020\\
    27.7763 0.0005\\
    35.3559 0.0\\
};
\addlegendentry{HOOI}

\end{axis}

\end{tikzpicture}%

%% file: tikzfig/hsi.tex
%
%
\begin{tikzpicture}

\begin{axis}[%
scale = 0.55,
width=4.521in,
height=3.566in,
at={(0.758in,0.481in)},
scale only axis,
xmode=log,
xmin=1,
xmax=8,
xminorticks=true,
xlabel style={font=\color{white!15!black}},
xlabel={time(s)},
ymode=log,
ymin=0.001,
ymax=1,
yminorticks=true,
ylabel style={font=\color{white!15!black}},
ylabel={$\text{relErr - relErr}_{\text{min}}$},
axis background/.style={fill=white},
xmajorgrids,
ymajorgrids,
legend style={legend cell align=left, align=left, at={(0.4,0.4)}, draw=white!15!black}
]
\addplot [color=black, dotted, line width=1.1pt, mark=triangle, mark options={solid, black}]
  table[row sep=crcr]{%
    1.0000 0.5657\\
    3.6915 0.5657\\
    3.8880 0.1028\\
    4.0203 0.0257\\
    4.1084 0.0166\\
    4.2129 0.0164\\
};
\addlegendentry{D-Tucker}

\addplot [color=black, dashed, line width=1.1pt, mark=o, mark options={solid, black}]
  table[row sep=crcr]{%
    1.0000 0.5657\\
    2.1727 0.1642\\
    3.3183 0.0239\\
    4.5067 0.0084\\
    5.6181 0.0046\\
    6.7246 0.0033\\
    7.8567 0.0026\\
};
\addlegendentry{RPCD}

\addplot [color=black, solid, line width=1.1pt, mark=square, mark options={solid, black}]
  table[row sep=crcr]{%
    1.0000 0.5657\\
    2.1633 0.1178\\
    3.3234 0.0046\\
    4.4962 0.0010\\
    5.6873 0.0005\\
};
\addlegendentry{RPCD+}

\addplot [color=black, loosely dotted, line width=1.1pt, mark=diamond, mark options={solid, black}]
  table[row sep=crcr]{%
    1.0000 0.5657\\
    2.1524 0.1029\\
    3.2733 0.0024\\
    4.4050 0.0001\\
    5.4820 0.0\\
};
\addlegendentry{HOOI}

\end{axis}

\end{tikzpicture}%

%% file: tikzfig/coil.tex
%
%
\begin{tikzpicture}

\begin{axis}[%
scale = 0.55,
width=4.521in,
height=3.566in,
at={(0.758in,0.481in)},
scale only axis,
xmode=log,
xmin=1,
xmax=7.7158212,
xminorticks=true,
xlabel style={font=\color{white!15!black}},
xlabel={time(s)},
ymode=log,
ymin=0.001,
ymax=1,
yminorticks=true,
ylabel style={font=\color{white!15!black}},
ylabel={$\text{relErr - relErr}_{\text{min}}$},
axis background/.style={fill=white},
xmajorgrids,
ymajorgrids,
legend style={legend cell align=left, align=left, at={(0.8,0.8)},draw=white!15!black}
]
\addplot [color=black, dotted, line width=1.1pt, mark=triangle, mark options={solid, black}]
  table[row sep=crcr]{%
    1.0000 0.6364\\
    6.0703 0.6364\\
    6.7619 0.6158\\
    7.0911 0.0090\\
    7.3524 0.0030\\
    7.6204 0.0028\\
};
\addlegendentry{D-Tucker}

\addplot [color=black, dashed, line width=1.1pt, mark=o, mark options={solid, black}]
  table[row sep=crcr]{%
    1.0000 0.6364\\
    1.2721 0.1342\\
    1.5279 0.0158\\
    1.7668 0.0040\\
    2.0176 0.0016\\
    2.2619 0.0007\\
};
\addlegendentry{RPCD}

\addplot [color=black, solid, line width=1.1pt, mark=square, mark options={solid, black}]
  table[row sep=crcr]{%
    1.0000 0.6364\\
    1.2594 0.1308\\
    1.4984 0.0009\\
    1.7400 0.0\\
};
\addlegendentry{RPCD+}

\addplot [color=black, loosely dotted, line width=1.1pt, mark=diamond, mark options={solid, black}]
  table[row sep=crcr]{%
    1.0000 0.6364\\
    1.2903 0.1109\\
    1.5542 0.0007\\
    1.7971 0.0000\\
};
\addlegendentry{HOOI}

\end{axis}

\end{tikzpicture}%

%% file: tikzfig/brainq_1.tex
%
%
\begin{tikzpicture}

\begin{axis}[%
scale = 0.55,
width=4.521in,
height=3.566in,
at={(0.758in,0.481in)},
scale only axis,
xmode=log,
xmin=1,
xmax=24,
xminorticks=true,
xlabel style={font=\color{white!15!black}},
xlabel={time(s)},
ymode=log,
ymin=0.001,
ymax=1,
yminorticks=true,
ylabel style={font=\color{white!15!black}},
ylabel={$\text{relErr - relErr}_{\text{min}}$},
axis background/.style={fill=white},
xmajorgrids,
ymajorgrids,
legend style={legend cell align=left, align=left, draw=white!15!black}
]
\addplot [color=black, dotted, line width=1.1pt, mark=triangle, mark options={solid, black}]
  table[row sep=crcr]{%
    1.0000 0.2291\\
    1.4184 0.2291\\
    1.4947 0.0246\\
    1.5437 0.0056\\
    1.5882 0.0032\\
    1.6360 0.0018\\
    1.6808 0.0016\\
};
\addlegendentry{D-Tucker}

\addplot [color=black, dashed, line width=1.1pt, mark=o, mark options={solid, black}]
  table[row sep=crcr]{%
    1.0000 0.2291\\
    1.1889 0.1089\\
    1.3629 0.0252\\
    1.5369 0.0125\\
    1.7171 0.0077\\
    1.9126 0.0054\\
    2.1296 0.0043\\
    2.3303 0.0033\\
};
\addlegendentry{RPCD}

\addplot [color=black, solid, line width=1.1pt, mark=square, mark options={solid, black}]
  table[row sep=crcr]{%
    1.0000 0.2291\\
    1.2233 0.1031\\
    1.4109 0.0090\\
    1.6220 0.0062\\
    1.8029 0.0046\\
    1.9991 0.0028\\
    2.1808 0.0010\\
    2.3679    0.0\\
};
\addlegendentry{RPCD+}

\addplot [color=black,  loosely dotted, line width=1.1pt, mark=diamond, mark options={solid, black}]
  table[row sep=crcr]{%
    1.0000 0.2291\\
    5.2417 0.0769\\
    9.1083 0.0104\\
    12.6568 0.0058\\
    16.5418 0.0043\\
    20.0706 0.0030\\
    23.5960 0.0023\\
};
\addlegendentry{HOOI}

\end{axis}

\end{tikzpicture}%

%% file: tikzfig/brainq_2.tex
%
%
\begin{tikzpicture}

\begin{axis}[%
scale = 0.55,
width=4.521in,
height=3.566in,
at={(0.758in,0.481in)},
scale only axis,
xmode=log,
xmin=1,
xmax=20,
xminorticks=true,
xlabel style={font=\color{white!15!black}},
xlabel={time(s)},
ymode=log,
ymin=0.001,
ymax=1,
yminorticks=true,
ylabel style={font=\color{white!15!black}},
ylabel={$\text{relErr - relErr}_{\text{min}}$},
axis background/.style={fill=white},
xmajorgrids,
ymajorgrids,
legend style={legend cell align=left, align=left, draw=white!15!black}
]
\addplot [color=black, dotted, line width=1.1pt, mark=triangle, mark options={solid, black}]
  table[row sep=crcr]{%
    1.0000 0.2300\\
    1.4267 0.2300\\
    1.5541 0.0268\\
    1.6439 0.0068\\
    1.7040 0.0066\\
};
\addlegendentry{D-Tucker}

\addplot [color=black, dashed, line width=1.1pt, mark=o, mark options={solid, black}]
  table[row sep=crcr]{%
    1.0000  0.2300\\ 
    1.1917  0.0953\\ 
    1.3796  0.0255\\ 
    1.5632  0.0078\\ 
    1.7716  0.0050\\ 
    1.9588  0.0040\\ 
    2.1381  0.0033\\ 
};
\addlegendentry{RPCD}

\addplot [color=black, solid, line width=1.1pt, mark=square, mark options={solid, black}]
  table[row sep=crcr]{%
    1.0000  0.2300\\ 
    1.2153  0.0902\\ 
    1.4310  0.0123\\ 
    1.6380  0.0051\\ 
    1.8411  0.0028\\ 
    2.0405  0.0012\\ 
    2.2360  0.0003\\ 
};
\addlegendentry{RPCD+}

\addplot [color=black, loosely dotted, line width=1.1pt, mark=diamond, mark options={solid, black}]
  table[row sep=crcr]{%
    1.0000  0.2300\\ 
    5.2322  0.0758\\ 
    9.0707  0.0094\\ 
    12.5983  0.0028\\ 
    16.1045  0.0009\\ 
    19.6788  0.0000\\ 
};
\addlegendentry{HOOI}

\end{axis}

\end{tikzpicture}%

%% file: tikzfig/brainq_3.tex
%
%
\begin{tikzpicture}

\begin{axis}[%
scale = 0.55,
width=4.521in,
height=3.566in,
at={(0.758in,0.481in)},
scale only axis,
xmode=log,
xmin=1,
xmax=20,
xminorticks=true,
xlabel style={font=\color{white!15!black}},
xlabel={time(s)},
ymode=log,
ymin=0.001,
ymax=1,
yminorticks=true,
ylabel style={font=\color{white!15!black}},
ylabel={$\text{relErr - relErr}_{\text{min}}$},
axis background/.style={fill=white},
xmajorgrids,
ymajorgrids,
legend style={legend cell align=left, align=left, draw=white!15!black}
]
\addplot [color=black, dotted, line width=1.1pt, mark=triangle, mark options={solid, black}]
  table[row sep=crcr]{%
  1.0000  0.2324\\ 
 1.4312  0.2324\\ 
 1.4989  0.0285\\ 
 1.5564  0.0102\\ 
 1.6071  0.0100\\ 
};
\addlegendentry{D-Tucker}

\addplot [color=black, dashed, line width=1.1pt, mark=o, mark options={solid, black}]
  table[row sep=crcr]{%
  1.0000  0.2324\\ 
 1.1920  0.1143\\ 
 1.3709  0.0230\\ 
 1.5526  0.0101\\ 
 1.7669  0.0060\\ 
 1.9470  0.0040\\ 
 2.1279  0.0027\\ 
 2.3087  0.0018\\ 
};
\addlegendentry{RPCD}

\addplot [color=black, solid, line width=1.1pt, mark=square, mark options={solid, black}]
  table[row sep=crcr]{%
  1.0000  0.2324\\ 
 1.2047  0.1076\\ 
 1.3961  0.0095\\ 
 1.5805  0.0044\\ 
 1.7653  0.0030\\ 
 1.9572  0.0024\\ 
};
\addlegendentry{RPCD+}

\addplot [color=black, loosely dotted, line width=1.1pt, mark=diamond, mark options={solid, black}]
  table[row sep=crcr]{%
  1.0000  0.2324\\ 
 4.9480  0.0803\\ 
 8.4771  0.0091\\ 
 12.0059  0.0012\\ 
 15.4978  0.0001\\ 
 19.0303  0.0000\\ 
};
\addlegendentry{HOOI}

\end{axis}

\end{tikzpicture}%

%% file: tikzfig/brainq_4.tex
%
%
\begin{tikzpicture}

\begin{axis}[%
scale = 0.55,
width=4.521in,
height=3.566in,
at={(0.758in,0.481in)},
scale only axis,
xmode=log,
xmin=1,
xmax=21,
xminorticks=true,
xlabel style={font=\color{white!15!black}},
xlabel={time(s)},
ymode=log,
ymin=0.001,
ymax=1,
yminorticks=true,
ylabel style={font=\color{white!15!black}},
ylabel={$\text{relErr - relErr}_{\text{min}}$},
axis background/.style={fill=white},
xmajorgrids,
ymajorgrids,
legend style={legend cell align=left, align=left, draw=white!15!black}
]
\addplot [color=black, dotted, line width=1.1pt, mark=triangle, mark options={solid, black}]
  table[row sep=crcr]{%
1.0000  0.2321\\ 
 1.4206  0.2321\\ 
 1.4972  0.0274\\ 
 1.5449  0.0084\\ 
 1.5950  0.0053\\ 
 1.6404  0.0046\\ 
};
\addlegendentry{D-Tucker}

\addplot [color=black, dashed, line width=1.1pt, mark=o, mark options={solid, black}]
  table[row sep=crcr]{%
1.0000  0.2320\\ 
 1.1794  0.0965\\ 
 1.3568  0.0281\\ 
 1.5675  0.0191\\ 
 1.7618  0.0164\\ 
 1.9500  0.0147\\ 
 2.1413  0.0134\\ 
 2.3313  0.0119\\ 
 2.5177  0.0098\\ 
 2.7062  0.0068\\ 
 2.8943  0.0039\\ 
 3.0817  0.0020\\ 
 3.2793  0.0009\\ 
 3.5124  0.0002\\ 
};
\addlegendentry{RPCD}

\addplot [color=black, solid, line width=1.1pt, mark=square, mark options={solid, black}]
  table[row sep=crcr]{%
1.0000  0.2320\\ 
 1.1978  0.0924\\ 
 1.4055  0.0253\\ 
 1.5945  0.0175\\ 
 1.7897  0.0155\\ 
 1.9866  0.0121\\ 
 2.1848  0.0058\\ 
 2.3707  0.0026\\ 
 2.5608  0.0021\\ 
};
\addlegendentry{RPCD+}

\addplot [color=black, loosely dotted, line width=1.1pt, mark=diamond, mark options={solid, black}]
  table[row sep=crcr]{%
1.0000  0.2320\\ 
 4.9343  0.0742\\ 
 8.3877  0.0229\\ 
 11.8847  0.0137\\ 
 15.7023  0.0059\\ 
 19.2165  0.0018\\ 
 22.9895  0.0003\\ 
 26.5623  0.0000\\ 
};
\addlegendentry{HOOI}

\end{axis}

\end{tikzpicture}%

%% file: tikzfig/brainq_5.tex
%
%
\begin{tikzpicture}

\begin{axis}[%
scale = 0.55,
width=4.521in,
height=3.566in,
at={(0.758in,0.481in)},
scale only axis,
xmode=log,
xmin=1,
xmax=37,
xminorticks=true,
xlabel style={font=\color{white!15!black}},
xlabel={time(s)},
ymode=log,
ymin=0.001,
ymax=1,
yminorticks=true,
ylabel style={font=\color{white!15!black}},
ylabel={$\text{relErr - relErr}_{\text{min}}$},
axis background/.style={fill=white},
xmajorgrids,
ymajorgrids,
legend style={legend cell align=left, align=left, draw=white!15!black}
]
\addplot [color=black, dotted, line width=1.1pt, mark=triangle, mark options={solid, black}]
  table[row sep=crcr]{%
  1.0000  0.2319\\ 
 1.4277  0.2319\\ 
 1.4929  0.0304\\ 
 1.5395  0.0040\\ 
 1.5899  0.0035\\ 
};
\addlegendentry{D-Tucker}

\addplot [color=black, dashed, line width=1.1pt, mark=o, mark options={solid, black}]
  table[row sep=crcr]{%
  1.0000  0.2319\\ 
 1.1932  0.1191\\ 
 1.3792  0.0328\\ 
 1.5578  0.0176\\ 
 1.7493  0.0109\\ 
 1.9456  0.0079\\ 
 2.1373  0.0064\\ 
 2.3269  0.0050\\ 
 2.5260  0.0036\\ 
 2.7168  0.0026\\ 
 2.9043  0.0019\\ 
};
\addlegendentry{RPCD}

\addplot [color=black, solid, line width=1.1pt, mark=square, mark options={solid, black}]
  table[row sep=crcr]{%
  1.0000  0.2319\\ 
 1.2150  0.1153\\ 
 1.4049  0.0247\\ 
 1.6252  0.0123\\ 
 1.8495  0.0066\\ 
 2.0624  0.0039\\ 
 2.2554  0.0027\\ 
 2.4507  0.0019\\ 
};
\addlegendentry{RPCD+}

\addplot [color=black, loosely dotted, line width=1.1pt, mark=diamond, mark options={solid, black}]
  table[row sep=crcr]{%
  1.0000  0.2319\\ 
 5.3392  0.0880\\ 
 8.9850  0.0240\\ 
 12.3958  0.0178\\ 
 15.7407  0.0141\\ 
 19.3982  0.0088\\ 
 23.0726  0.0055\\ 
 26.7807  0.0040\\ 
 30.1795  0.0020\\ 
 33.5743  0.0007\\ 
 36.9935  0.0000\\ 
};
\addlegendentry{HOOI}

\end{axis}

\end{tikzpicture}%

%% file: tikzfig/brainq_6.tex
%
%
\begin{tikzpicture}

\begin{axis}[%
scale = 0.55,
width=4.521in,
height=3.566in,
at={(0.758in,0.481in)},
scale only axis,
xmode=log,
xmin=1,
xmax=30,
xminorticks=true,
xlabel style={font=\color{white!15!black}},
xlabel={time(s)},
ymode=log,
ymin=0.001,
ymax=1,
yminorticks=true,
ylabel style={font=\color{white!15!black}},
ylabel={$\text{relErr - relErr}_{\text{min}}$},
axis background/.style={fill=white},
xmajorgrids,
ymajorgrids,
legend style={legend cell align=left, align=left, draw=white!15!black}
]
\addplot [color=black, dotted, line width=1.1pt, mark=triangle, mark options={solid, black}]
  table[row sep=crcr]{%
  1.0000  0.2301\\ 
 1.4108  0.2301\\ 
 1.4882  0.0240\\ 
 1.5370  0.0014\\ 
 1.5796  0.0010\\ 
};
\addlegendentry{D-Tucker}

\addplot [color=black, dashed, line width=1.1pt, mark=o, mark options={solid, black}]
  table[row sep=crcr]{%
  1.0000  0.2301\\ 
 1.1782  0.1034\\ 
 1.3573  0.0249\\ 
 1.5629  0.0201\\ 
 1.7487  0.0171\\ 
 1.9274  0.0136\\ 
 2.1097  0.0108\\ 
 2.2957  0.0081\\ 
 2.4735  0.0049\\ 
 2.6565  0.0023\\ 
 2.8545  0.0008\\ 
 3.0442  0.0000\\ 
};
\addlegendentry{RPCD}

\addplot [color=black, solid, line width=1.1pt, mark=square, mark options={solid, black}]
  table[row sep=crcr]{%
  1.0000  0.2301\\ 
 1.2087  0.0986\\ 
 1.4053  0.0204\\ 
 1.6052  0.0163\\ 
 1.8203  0.0123\\ 
 2.0469  0.0082\\ 
 2.2596  0.0048\\ 
 2.4611  0.0026\\ 
 2.6485  0.0011\\ 
 2.8455  0.0002\\ 
};
\addlegendentry{RPCD+}

\addplot [color=black, loosely dotted, line width=1.1pt, mark=diamond, mark options={solid, black}]
  table[row sep=crcr]{%
  1.0000  0.2301\\ 
 4.9042  0.0733\\ 
 8.4181  0.0178\\ 
 11.9326  0.0125\\ 
 15.7991  0.0078\\ 
 19.6665  0.0042\\ 
 23.1899  0.0020\\ 
 26.6999  0.0007\\ 
 30.2090  0.0000\\ 
};
\addlegendentry{HOOI}

\end{axis}

\end{tikzpicture}%

%% file: rpcd.bbl
\begin{thebibliography}{10}

\bibitem{absil2009optimization}
{\sc P.-A. Absil, R.~Mahony, and R.~Sepulchre}, {\em Optimization algorithms on
  matrix manifolds}, Princeton University Press, Princeton, NJ, United States,
  2009.

\bibitem{austin2016parallel}
{\sc W.~Austin, G.~Ballard, and T.~G. Kolda}, {\em Parallel tensor compression
  for large-scale scientific data}, in IEEE International parallel and
  distributed processing symposium (IPDPS), 2016, pp.~912--922.

\bibitem{beck2013convergence}
{\sc A.~Beck and L.~Tetruashvili}, {\em On the convergence of block coordinate
  descent type methods}, SIAM Journal on Optimization, 23 (2013),
  pp.~2037--2060.

\bibitem{belhumeur1997eigenfaces}
{\sc P.~N. Belhumeur, J.~P. Hespanha, and D.~J. Kriegman}, {\em Eigenfaces vs.
  {Fisherfaces}: Recognition using class specific linear projection}, IEEE
  Transactions on Pattern Analysis and Machine Intelligence, 19 (1997),
  pp.~711--720.

\bibitem{boumal2020introduction}
{\sc N.~Boumal}, {\em An introduction to optimization on smooth manifolds},
  Cambridge University Press, Cambridge, United Kingdom, 2023.

\bibitem{boumal2019global}
{\sc N.~Boumal, P.-A. Absil, and C.~Cartis}, {\em Global rates of convergence
  for nonconvex optimization on manifolds}, IMA Journal of Numerical Analysis,
  39 (2019), pp.~1--33.

\bibitem{carlini2011ranks}
{\sc E.~Carlini and J.~Kleppe}, {\em Ranks derived from multilinear maps},
  Journal of Pure and Applied Algebra, 215 (2011), pp.~1999--2004.

\bibitem{che2019randomized}
{\sc M.~Che and Y.~Wei}, {\em Randomized algorithms for the approximations of
  {Tucker} and the tensor train decompositions}, Advances in Computational
  Mathematics, 45 (2019), pp.~395--428.

\bibitem{cichocki2009nonnegative}
{\sc A.~Cichocki, R.~Zdunek, A.~H. Phan, and S.-i. Amari}, {\em Nonnegative
  matrix and tensor factorizations: applications to exploratory multi-way data
  analysis and blind source separation}, John Wiley \& Sons, West Sussex,
  United Kingdom, 2009.

\bibitem{de2000multilinear}
{\sc L.~De~Lathauwer, B.~De~Moor, and J.~Vandewalle}, {\em A multilinear
  singular value decomposition}, SIAM Journal on Matrix Analysis and
  Applications, 21 (2000), pp.~1253--1278.

\bibitem{de2000best}
{\sc L.~De~Lathauwer, B.~De~Moor, and J.~Vandewalle}, {\em On the best rank-$1$
  and rank-$(r_1, r_2,..., r_n)$ approximation of higher-order tensors}, SIAM
  journal on Matrix Analysis and Applications, 21 (2000), pp.~1324--1342.

\bibitem{elden2009newton}
{\sc L.~Eld{\'e}n and B.~Savas}, {\em A {Newton-Grassmann} method for computing
  the best multilinear rank-$(r_1, r_2, r_3)$ approximation of a tensor}, SIAM
  Journal on Matrix Analysis and applications, 31 (2009), pp.~248--271.

\bibitem{foster2006frequency}
{\sc D.~H. Foster, K.~Amano, S.~M. Nascimento, and M.~J. Foster}, {\em
  Frequency of metamerism in natural scenes}, Journal of the Optical Society of
  America A, 23 (2006), pp.~2359--2372.

\bibitem{golub1996matrix}
{\sc G.~H. Golub and C.~F. Van~Loan}, {\em Matrix computations}, Johns Hopkins
  University Press, Baltimore, MD, United States, 1996.

\bibitem{grasedyck2010hierarchical}
{\sc L.~Grasedyck}, {\em Hierarchical singular value decomposition of tensors},
  SIAM Journal on Matrix Analysis and Applications, 31 (2010), pp.~2029--2054.

\bibitem{gutman2022coordinate}
{\sc D.~H. Gutman and N.~Ho-Nguyen}, {\em Coordinate descent without
  coordinates: Tangent subspace descent on {Riemannian} manifolds}, Mathematics
  of Operations Research,  (2022).

\bibitem{ishteva2011best}
{\sc M.~Ishteva, P.-A. Absil, S.~Van~Huffel, and L.~De~Lathauwer}, {\em Best
  low multilinear rank approximation of higher-order tensors, based on the
  {Riemannian} trust-region scheme}, SIAM Journal on Matrix Analysis and
  Applications, 32 (2011), pp.~115--135.

\bibitem{jang2020d}
{\sc J.-G. Jang and U.~Kang}, {\em {D-Tucker}: Fast and memory-efficient tucker
  decomposition for dense tensors}, in IEEE International Conference on Data
  Engineering (ICDE), 2020, pp.~1850--1853.

\bibitem{kasai2016low}
{\sc H.~Kasai and B.~Mishra}, {\em Low-rank tensor completion: a {Riemannian}
  manifold preconditioning approach}, in International Conference on Machine
  Learning (ICML), 2016, pp.~1012--1021.

\bibitem{kolda2009tensor}
{\sc T.~G. Kolda and B.~W. Bader}, {\em Tensor decompositions and
  applications}, SIAM Review, 51 (2009), pp.~455--500.

\bibitem{koldatoolbox}
{\sc T.~G. Kolda and B.~W. Bader}, {\em Tensor toolbox for {MATLAB}, version
  3.2.1}, 2021, \url{https://www.tensortoolbox.org}.

\bibitem{kressner2014low}
{\sc D.~Kressner, M.~Steinlechner, and B.~Vandereycken}, {\em Low-rank tensor
  completion by {Riemannian} optimization}, BIT Numerical Mathematics, 54
  (2014), pp.~447--468.

\bibitem{lu_plataniotis_venetsanopoulos_2008}
{\sc H.~Lu, K.~Plataniotis, and A.~Venetsanopoulos}, {\em {MPCA}: Multilinear
  principal component analysis of tensor objects}, IEEE Transactions on Neural
  Networks, 19 (2008), pp.~18--39.

\bibitem{mishra2016riemannian}
{\sc B.~Mishra and R.~Sepulchre}, {\em {Riemannian} preconditioning}, SIAM
  Journal on Optimization, 26 (2016), pp.~635--660.

\bibitem{mitchell2008predicting}
{\sc T.~M. Mitchell, S.~V. Shinkareva, A.~Carlson, K.-M. Chang, V.~L. Malave,
  R.~A. Mason, and M.~A. Just}, {\em Predicting human brain activity associated
  with the meanings of nouns}, Science, 320 (2008), pp.~1191--1195.

\bibitem{morup2008algorithms}
{\sc M.~M{\o}rup, L.~K. Hansen, and S.~M. Arnfred}, {\em Algorithms for sparse
  nonnegative {Tucker} decompositions}, Neural Computation, 20 (2008),
  pp.~2112--2131.

\bibitem{nene1996columbia}
{\sc S.~A. Nene, S.~K. Nayar, and H.~Murase}, {\em Columbia object image
  library ({COIL}-100)}, Tech. Report CUCS-006-96, Department of Computer
  Science, Columbia University, 1996.

\bibitem{oh2017s}
{\sc J.~Oh, K.~Shin, E.~E. Papalexakis, C.~Faloutsos, and H.~Yu}, {\em S-hot:
  Scalable high-order {Tucker} decomposition}, in ACM International Conference
  on Web Search and Data Mining (WSDM), 2017, pp.~761--770.

\bibitem{savas2010quasi}
{\sc B.~Savas and L.-H. Lim}, {\em Quasi-{Newton} methods on {Grassmannians}
  and multilinear approximations of tensors}, SIAM Journal on Scientific
  Computing, 32 (2010), pp.~3352--3393.

\bibitem{sidiropoulos2017tensor}
{\sc N.~D. Sidiropoulos, L.~De~Lathauwer, X.~Fu, K.~Huang, E.~E. Papalexakis,
  and C.~Faloutsos}, {\em Tensor decomposition for signal processing and
  machine learning}, IEEE Transactions on Signal Processing, 65 (2017),
  pp.~3551--3582.

\bibitem{sun2020low}
{\sc Y.~Sun, Y.~Guo, C.~Luo, J.~Tropp, and M.~Udell}, {\em Low-rank {Tucker}
  approximation of a tensor from streaming data}, SIAM Journal on Mathematics
  of Data Science, 2 (2020), pp.~1123--1150.

\bibitem{tappenden2016inexact}
{\sc R.~Tappenden, P.~Richt{\'a}rik, and J.~Gondzio}, {\em Inexact coordinate
  descent: complexity and preconditioning}, Journal of Optimization Theory and
  Applications, 170 (2016), pp.~144--176.

\bibitem{tucker1966some}
{\sc L.~R. Tucker}, {\em Some mathematical notes on three-mode factor
  analysis}, Psychometrika, 31 (1966), pp.~279--311.

\bibitem{vannieuwenhoven2012new}
{\sc N.~Vannieuwenhoven, R.~Vandebril, and K.~Meerbergen}, {\em A new
  truncation strategy for the higher-order singular value decomposition}, SIAM
  Journal on Scientific Computing, 34 (2012), pp.~A1027--A1052.

\bibitem{wright2015coordinate}
{\sc S.~J. Wright}, {\em Coordinate descent algorithms}, Mathematical
  Programming, 151 (2015), pp.~3--34.

\bibitem{xu2018convergence}
{\sc Y.~Xu}, {\em On the convergence of higher-order orthogonal iteration},
  Linear and Multilinear Algebra, 66 (2018), pp.~2247--2265.

\bibitem{zare2018extension}
{\sc A.~Zare, A.~Ozdemir, M.~A. Iwen, and S.~Aviyente}, {\em Extension of {PCA}
  to higher order data structures: An introduction to tensors, tensor
  decompositions, and tensor {PCA}}, Proceedings of the IEEE, 106 (2018),
  pp.~1341--1358.

\end{thebibliography}
